\documentclass{article}
\usepackage[utf8]{inputenc}
\usepackage{tabularx} 
\usepackage{amsmath, amsthm}  
\usepackage{amsmath, amssymb, bm, mathtools, mathdots}
\numberwithin{equation}{section}
\usepackage{floatrow}
\usepackage{bbm}
\usepackage{tikz-cd}
\usepackage{graphicx} 
\usepackage[margin=1in,letterpaper]{geometry} 
\usepackage{cite} 
\usepackage[all]{xy}
\usepackage{tikz}
\usetikzlibrary{arrows.meta}
\usetikzlibrary{commutative-diagrams}
\usepackage[final]{hyperref} 
\hypersetup{
	colorlinks=true,       
	linkcolor=blue,        
	citecolor=blue,        
	filecolor=magenta,     
	urlcolor=blue         
}
\usepackage{xcolor}
\newcommand{\R}{\mathbb{R}} 

\newcommand{\ud}{\,\mathrm{d}}
\newcommand{\udiv}{\, \mathrm{div}}

\newtheorem{theorem}{Theorem}[section]  
\newtheorem{definition}{Definition}

\newtheorem{proposition}{Proposition}[section]
\newtheorem{lemma}{Lemma}[section]
\newtheorem{corollary}{Corollary}[section]
\newtheorem{remark}{Remark}[section]

\newtheorem{claim}{Claim}

\title{Propagation of Chaos for 2D Log Gas on the Whole Space}

\author{Shuzhe Cai \footnote{Beijing International Center for Mathematical Research, Peking University, Beijing 100871, China. 2206392386@pku.edu.cn} 
	\quad
	Xuanrui Feng \footnote{Beijing International Center for Mathematical Research, Peking University, Beijing 100871,  China. pkufengxuanrui@stu.pku.edu.cn} \quad 
	Yun Gong \footnote{Beijing International Center for Mathematical Research, Peking University, Beijing 100871,  China.  1901110013@pku.edu.cn.}
    \quad
    Zhenfu Wang\footnote{Beijing International Center for Mathematical Research, Peking University, Beijing 100871,  China. zwang@bicmr.pku.edu.cn.}}

\begin{document}

\maketitle

\begin{abstract}
We derive the quantitative propagation of chaos in the sense of relative entropy for the first time for the 2D Log gas or the weakly interacting particle systems with 2D Coulomb interactions on the whole space. We resolve this problem by adapting the modulated free energy method in \cite{bresch2023mean} to the whole space setting and  establishing the crucial logarithmic growth estimates for the mean-field Poisson–Nernst–Planck (PNP) equation of single component via the parabolic maximum principle. 
\end{abstract}

\tableofcontents

\section{Introduction}

We consider the following large system of $N$ indistinguishable particles on the two dimensional space $\R^2$, whose evolution is given by the stochastic differential equations (SDEs)
\begin{equation}\label{Eq:particle system}
\begin{aligned}
\ud X_t^i = \frac{1}{N} \sum_{j \neq i} K(X_t^i - X_t^j)\ud t + \sqrt{\frac{2}{\beta}} \ud W_t^i, \ \ \ i = 1,...,N,
\end{aligned}
\end{equation}
where $X_t^i$ denotes the position of the $i$-th particle at time $t$, and $\lbrace W_t^i \rbrace_{1 \leq i \leq N}$ are $N$ independent copies of standard Brownian motions on $\R^2$, and $K:\mathbb{R}^2 \to \mathbb{R}^2$ denotes the 2-body interaction force between any two particles, say $X_i$ and $X_j$. The movement of each particle is governed by a deterministic drift term in the form of pairwise interacting force $K$ and a stochastic diffusion term given by the additive noise. The inverse temperature  $\beta$ is a fixed  positive number that represents the strength of the stochastic term of the particle system. 
In this article, we focus on the mean field limit of the 2D Coulomb or Log gas. So we simply choose $K=-\nabla g$, with $g(x) = -\frac{1}{2 \pi} \log|x|$  the fundamental solution of the Possion equation in $\mathbb{R}^2$, namely $-\Delta g=\delta_0$. Now the interacting kernel $K$ reads as 
\begin{equation}\label{Kdef}
K(x) = -\nabla g(x) =  \frac{1}{2 \pi} \frac{x}{|x|^2},
\end{equation}
where $x = (x_1, x_2) \in \R^2$. The interaction force $K$ is of replusive type and has  singularity  of order $O(1/|x|)$ when $x$ approaches 0.  

We are interested in deriving the mean-field limit  of the particle system \eqref{Eq:particle system} when $N \rightarrow \infty$. By classical results of mean-field limit, typically for systems with regular interaction forces $K\in W^{1, \infty}$ for instance in \cite{dobrushin1979vlasov,mckean1967propagation}, we expect that each particle $X_t^i$ of the particle system \eqref{Eq:particle system} converges to the following McKean-Vlasov process $\bar{X}_t$  as $N \rightarrow \infty$:
\begin{equation}\label{Eq:McKean-Vlasov}
\ud \bar{X}_t = K \ast \bar{\rho}_t(\bar{X}_t) \ud t + \sqrt{\frac{2}{\beta}} \ud W_t, \ \ \ \ \text{ where } \bar{\rho}_t = \text{Law}(\bar{X}_t), 
\end{equation}
where $\bar{X}_t \in \R^2$ and $W_t$ is a standard Brownian motion on $\R^2$ independent of $\lbrace W_t^i \rbrace_{1 \leq i \leq N}$. The process above is self-consistent since  the evolution of $\bar X_t$ depends on the distribution/law $\bar{\rho}_t$ of itself. By applying  It\^o's formula, we deduce that the law $\bar{\rho}_t$ of $\bar{X_t}$ satisfies the nonlinear Fokker-Planck equation or also McKean-Vlasov PDE
\begin{equation}\label{Eq:limiting equation}
\partial_t \bar{\rho} + \text{div}_x \Big(\bar{\rho}\,  K \ast \bar{\rho}  \Big) = \frac{1}{\beta} \Delta_x \bar{\rho}.
\end{equation}
When we take $K$ as the 2D Coulomb force, the McKean-Vlasov PDE \eqref{Eq:limiting equation} is also called the mean-field Poisson–Nernst–Planck (PNP) of single component. 

The convergence from the particle  system \eqref{Eq:particle system} towards the  McKean-Vlasov system \eqref{Eq:McKean-Vlasov}  is usually termed as mean field limit and can be described precisely by {\em Propagation of Chaos} in mathematics literature. In our 2D Log or 2D Coulomb case, quantitative propagation of chaos for systems set on the torus $\mathbb{T}^2$ has been established by Bresch-Jabin-Wang \cite{bresch2019modulated}, while qualitative propagation of chaos (without convergence rate) for 2D Log gas  on the whole space $\mathbb{R}^2$  has been established by Liu-Yang \cite{Liu2016PropagationOC}, following the entropy based compactness argument in \cite{fournier2014propagation}. For more detailed discussions concerning the mean-field theory of various models, we refer to the  reviews such as \cite{sznitman1991topics,golse2016dynamics,jabin2014review,jabin2017mean}. 

\subsection{Modulated Free Energy}

We shall apply the so-called modulated free energy method introduced in \cite{bresch2019modulated,bresch2023mean} to derive the propagation of chaos result for the particle system \eqref{Eq:particle system}. Since our method will work on the statistical sense, we only care about the evolution of the joint distribution  $\rho_N(t, \cdot) $ of the particle system \eqref{Eq:particle system}, that is  the Liouville equation (forward Kolmogorov equation) 
\begin{equation}\label{Eq:Liouville}
\partial_t \rho_N + \sum_{i=1}^N \text{div}_{x_i} \bigg(\rho_N \frac{1}{N} \sum_{j \neq i} K(x_i - x_j)\bigg) = \frac{1}{\beta} \sum_{i=1}^N \Delta_{x_i} \rho_N, 
\end{equation}
where  $\rho_N = \rho_N(t,x_1,...,x_N)$ is the joint distribution of $(X_t^1, X_t^2, \cdots, X_t^N)$.   Since by assumption all the particles in \eqref{Eq:particle system} are indistinguishable, once we  assume that the initial law $\rho_N(0,\cdot)$ is a symmetric probability measure on $\R^{2N}$, i.e. $\rho_N(0,\cdot) \in \mathcal{P}_{sym}(\R^{2N})$, then so is $\rho_N(t,\cdot)$ for any $t > 0$. It is also convenient to define the following tensorized law associated with the mean-field equation \eqref{Eq:limiting equation}:
\begin{equation*}
\bar{\rho}_N(t,x_1,...,x_N) := \bar \rho_t^{\otimes N}(x_1, \cdots, x_N) = \prod_{i=1}^N \bar{\rho}(t,x_i).
\end{equation*}
The main idea of this paper relies on propagating and controlling the modulated free energy between the joint distribution $\rho_N$ or  the empirical measure $\mu_N $ under the law $\rho_N$, i.e. 
\begin{equation*}
    \mu_N(x)=\frac{1}{N}\sum_{i=1}^N \delta_{X_t^i}(x), 
\end{equation*}
and the limiting density $\bar\rho$, which combines the relative entropy introduced by Jabin-Wang \cite{jabin2018quantitative} and the modulated energy introduced by Serfaty \cite{serfaty2020mean} and Duerinckx \cite{Duerinckx2015MeanFieldLF}. Roughly speaking, the relative entropy method can make perfect use of the diffusion term with positive viscosity, as when computing the time derivative of the relative entropy, it produces an extra negative relative Fisher information term. However, one may face some singular term $\udiv K(x-y)$ in the computations, which naturally vanishes in the 2D Navier-Stokes case, since in that case the kernel given by Biot-Savart law $$K=\frac{1}{2\pi} \nabla^\perp \log |x|$$ is divergence-free. On the other hand, the modulated energy method can well treat the Coulomb case in any dimension but without diffusion term, since it is in its nature designed as a Coulomb/Riesz based metric in view of Fourier transform. A pure modulated energy method is also not enough to show strong propagation of chaos. Hence in the series of work by Bresch-Jabin-Wang \cite{bresch2019mean,bresch2019modulated,bresch2023mean}, the above two quantites are combined with appropriate weights, which cancels exactly the singular term $\udiv K(x-y)$ in the computation of time derivative, but still preserves the  coerciveness over the relative entropy. 

Now we introduce the modulated free energy $E_N(\rho_N|\bar{\rho}_N)$ on $\R^d$ with general dimension $d$ as follows, 
\begin{equation}\label{Quantity:modulated free energy-1}
\begin{aligned}
E_N(\rho_N|\bar{\rho}_N) = \mathcal{H}_N(\rho_N|\bar{\rho}_N) + \frac{1}{\beta} \int_{\R^{dN}} F_N(\mu_N, \bar{\rho}) \rho_N \ud X^N,
\end{aligned}
\end{equation}
where $X^N=(x_1,...,x_N)$ and $\mathcal{H}_N(\rho_N|\bar{\rho}_N)$ is exactly the normalized relative entropy between the joint law $\rho_N$ and the tensorized law $\bar\rho_N$ on the whole space:
\begin{equation}\label{Quantity:Relative entropy}
\mathcal{H}_N(\rho_N|\bar{\rho}_N) = \frac{1}{N} \int_{\R^{dN}} \rho_N \log \frac{\rho_N}{\bar{\rho}_N} \ud X^N,
\end{equation}
and $F_N(\mu_N,\bar\rho)$ is the modulated energy between the empirical measure $\mu_N$ and the limit density $\bar\rho$:
\begin{equation}\label{Quantity: modulated energy}
F_N(\mu_N, \bar\rho) = \int_{(\R^{d})^2 \backslash \Delta} g(x - y)(\ud \mu_N - \ud \bar{\rho})^{\otimes 2}(x,y), 
\end{equation}
here $\Delta$ denotes the diagonal in $\R^d \times \R^d$ and we remove this domain in our integral. The quantity $E_N(\rho_N | \bar{\rho}_N)$ can also be understood as the normalized relative entropy with two weights related to the Gibbs equilibrium, i.e.
\begin{equation}\label{Weighted relative entropy}
E_N(\rho_N|\bar{\rho}_N) = \frac{1}{N} \int_{\R^{dN}} \rho_N \log \bigg( \frac{\rho_N(t, X^N)}{G_N(t, X^N)} \frac{G_{\bar{\rho}_N}(t, X^N)}{\bar{\rho}_N(t, X^N
)}\bigg) \ud X^N,
\end{equation}
where
\begin{equation*}
\begin{aligned}
& G_N(t,X^N) = \exp \bigg( - \frac{\beta}{2N} \sum_{i \neq j} g(x_i - x_j) \bigg), \\
& G_{\bar{\rho}}(t,x) = \exp \bigg( - \beta g \ast \bar{\rho}(x) + \frac{\beta}{2} \int_{\R^{dN}} g \ast \bar{\rho} \bar{\rho} \bigg), \\
& G_{\bar{\rho}_N}(t,X^N) = \exp \bigg( - \beta \sum_{i=1}^N g \ast \bar{\rho}(x_i) + \frac{\beta N}{2} \int_{\R^{dN}} g \ast \bar{\rho} \bar{\rho} \bigg).
\end{aligned}
\end{equation*}
This kind of formula has been successfully used to derive the propagation of chaos for repulsive kernels and mild attractive kernels (including the 2D Patlak-Keller-Segel system) in the series of works \cite{bresch2019mean,bresch2019modulated,bresch2023mean}. 

\subsection{Main Results}

Based on the quantities we use above, we first need some assumptions on the solutions of the Liouville equation \eqref{Eq:Liouville} to justify our calculations below. Hence we introduce the notion of entropy solutions.
\begin{definition}[Entropy Solution] Let $T > 0$ be fixed. A density $\rho_N \in L^{\infty}([0,T]; L^1(\R^{dN}))$ with $\rho_N \geq 0$ and $\int_{\R^{dN}} \rho_N \ud X^N = 1$ is an entropy solution to the Liouville equation \eqref{Eq:Liouville} on the time interval $[0,T]$, if it solves \eqref{Eq:Liouville} in the sense of distribution, and for a.e. $t \leq T$,  
\begin{equation}\label{IEq:entropy inequality}
\begin{aligned}
&\int_{\R^{dN}} \rho_N(t, X^N) \log \bigg( \frac{\rho_N(t, X^N)}{G_N}\bigg) \ud X^N+\frac{1}{\beta} \sum_{i=1}^N \int_0^t \int_{\R^{dN}} \rho_N(s, X^N) \left| \nabla_{x_i} \log \bigg( \frac{\rho_N(s, X^N)}{G_N} \bigg) \right|^2 \ud X^N \ud s\\
\leq &\, \int_{\R^{dN}} \rho^0_N \log \bigg( \frac{\rho_N^0}{G_N}\bigg) \ud X^N. 
\end{aligned}
\end{equation}
\end{definition}
The existence of such entropy solution of the Liouville equation \eqref{Eq:Liouville} is given by the following proposition.
\begin{proposition}
If $\rho_N^0$ is a probability density on $\R^{dN}$ such that
\begin{equation*}
\int_{\R^{dN}} \rho_N^0 \log \bigg(\frac{\rho_N^0}{G_N}\bigg) \ud X^N < \infty,
\end{equation*}
then there exists a global entropy solution to \eqref{Eq:Liouville} with initial data $\rho_N^0$.
\end{proposition}
\begin{proof}
The existence of weak solutions relies on solving the regularization of the Liouville equation \eqref{Eq:Liouville}, denoting the corresponding classical solutions by $\{\rho_N^{\varepsilon}(t,X^N)\}_{\varepsilon>0}$ and showing that this sequence of densities are relatively compact. We refer to \cite[Theorem 2.1]{Liu2016PropagationOC} for more details. Then we can directly check that regularized sequence $\{\rho_N^{\varepsilon}(t,X^N)\}_{\varepsilon>0}$ satisfies the entropy dissipation inequality \eqref{IEq:entropy inequality} in the definition of entropy solution, and we obtain the desired result by taking limit of $\{\rho_N^{\varepsilon}(t,X^N)\}_{\varepsilon>0}$ as $\varepsilon \rightarrow 0$. 
\end{proof}

Now we are able to state our main result. Briefly speaking, our main contribution is to extend this modulated free energy method, which has been used in \cite{bresch2019modulated} for the 2D Coulomb gas on the torus case, to the 2D whole space by obtaining the logarithimic growth estimates for the solution of the limit equation \eqref{Eq:limiting equation}, which leads to the first quantitative convergence from the $N$-particle system \eqref{Eq:particle system} towards the limit McKean-Vlasov system \eqref{Eq:McKean-Vlasov} on $\R^2$ with any $\beta > 0$. This result also improves the qualitative convergence from the particle system \eqref{Eq:particle system} towards the limit system \eqref{Eq:McKean-Vlasov} on $\R^2$ by the compactness argument in \cite{Liu2016PropagationOC} to a quantitative counterpart. We finally remark that this method temporarily cannot cover the particle system \eqref{Eq:particle system} in $\mathbb{R}^d$  with dimension $d \geq 3$, where the control of time evolution of the modulated free energy is much more complicated than the case of $d = 2$. We will leave propagation of chaos for 3D or more higher dimensional Coulomb gas on the whole space to our further research.

\begin{theorem}[Entropic Propagation of Chaos]\label{Thm:EPoC}
Assume that $\rho_N$ is an entropy solution to the Liouville equation \eqref{Eq:Liouville} and that $\bar{\rho} \in L^{\infty}([0,T], L^1 \cap L^{\infty}(\R^2))$ solves the limit equation \eqref{Eq:limiting equation} with $\bar{\rho} \geq 0$ and $\int_{\R^2} \bar{\rho}(x,t)dx = 1$. Assume that the initial data $\bar{\rho}_0 \in W^{2,1} \cap W^{2,\infty}(\R^2)$ satisfies the logarithmic growth conditions
\begin{equation}\label{Ass:nabla-log}
\begin{aligned}
|\nabla \log \bar{\rho}_0(x)| &\lesssim 1+|x|, \\
|\nabla^2 \log \bar{\rho}_0(x)| &\lesssim 1+|x|^2,
\end{aligned}
\end{equation}
and the Gaussian upper bound
\begin{equation}\label{Ass:Gaussian-upper-bound}
\bar{\rho}_0(x) \leq C_0 \exp(-C_0^{-1}|x|^2),
\end{equation}
for some constant $C_0>0$. Then we have
\begin{equation*}
\mathcal{H}_N(\rho_N|\bar{\rho}_N)(t) \leq C e^{\frac{C}{\varepsilon}t^{\varepsilon}}\Big(E_N(\rho_N|\bar{\rho}_N)(0) + \frac{1}{N}\Big), 
\end{equation*}
where $C = C(\beta, \bar{\rho}_0, C_0) > 0$ is some positive constant and $\varepsilon > 0$ is any small positive constant.
\end{theorem}

\begin{remark}
We remark that for the 2D viscous vortex model, it shares the same result with the form 
\begin{equation*}
\mathcal{H}_N(\rho_N|\bar{\rho}_N)(t) \leq C e^{\frac{C}{\varepsilon}t^{\varepsilon}}\Big(\mathcal{H}_N(\rho_N|\bar{\rho}_N)(0) + \frac{1}{N}\Big)
\end{equation*}
on the whole space, which refines the time growth estimate compared to the previous work Feng-Wang \cite{feng2023quantitative}.
We can obtain this result by our new time decay estimates of $|\nabla \log \bar{\rho}|$ and $|\nabla^2 \log \bar{\rho}|$ in Section \ref{Hamilton type gradient estimates}.  
\end{remark}

\begin{remark}
We expect that the best time growth estimates in Theorem \ref{Thm:EPoC} should be polynomial in $t$, i.e.
\begin{equation*}
\mathcal{H}_N(\rho_N|\bar{\rho}_N)(t) \leq C (1+t)^c\Big(E_N(\rho_N|\bar{\rho}_N)(0) + \frac{1}{N}\Big), 
\end{equation*}
where the constant $c = c(\beta, \bar{\rho}_0, C_0) > 0$. This shares the same rate in $t$ as the sub-Coulombian cases when $d \geq 3$, which have been proved in \cite[Theorem 1.2]{rosenzweig2023global}. The reason why we cannot reach this optimal time growth esitmate is due to the suboptimal time decay estimates of $|\nabla \log \bar{\rho}|$ and $|\nabla^2 \log \bar{\rho}|$ by the parabolic maximum principle, see Theorem \ref{Thm:Logarithm gradient estimate}, Theorem \ref{Thm:Logarithm Hessian estimate} and Remark \ref{Rmk: optimal log estimates} in Section \ref{Hamilton type gradient estimates} for more details.  
\end{remark}

 In order to study uniform-in-time propagation of chaos, one need first establish the concentration of the equilibrium measure  (the Gibbs measure) of $N$-particle system  \eqref{Eq:particle system} around the equlibrium of the mean field equation \eqref{Eq:McKean-Vlasov}. However, for the Coulomb or Riesz gas in any dimension $\mathbb{R}^d$, this problem is relatively easy if one uses relative entropy carefully.    For the cases of Coulombian and Riesz interaction potentials,  the  equilibria for particle systems \eqref{Eq:particle system}  are given by the Gibbs measures, i.e. for all $d \geq 2$, 
\begin{equation}\label{Quantity: equilibrim of PS}
\rho_{N, \infty} = \frac{1}{Z_N} e^{-\beta \sum_{i \neq j} g(x_i - x_j)}, 
\end{equation}
where $Z_N$ is the partition function associated with $\rho_{N,\infty}$, i.e.
\begin{equation*}
Z_N = \int_{\R^{dN}} e^{-\beta \sum_{i \neq j} g(x_i - x_j)} \ud X_N,   
\end{equation*}
and $g(x): \R^d \rightarrow \R$ is defined as
\begin{equation*}
g(x) = \left\{ 
\begin{aligned}
& - \log|x|, \ \ & s = 0, \\
& \frac{1}{|x|^s},\ \ & 0 < s < d,
\end{aligned}
\right.
\end{equation*}
here we set $g(0) = 0$ for convenience. We obtain  the system \eqref{Eq:particle system} when we take $d = 2$ and $s = 0$. Corresponding to the equilibrium \eqref{Quantity: equilibrim of PS}, we also have the following formal equilibrium for limiting equation \eqref{Eq:limiting equation} for all $d \geq 2$ and $0 \leq s < d$,
\begin{equation}\label{Quantity: limiting equilibrium}
\bar{\rho}_{\infty} = \frac{1}{Z} e^{- \beta g \ast \bar{\rho}_{\infty}}, \ \ \ \ \ Z = \int_{\R^d} e^{- \beta g \ast \bar{\rho}_{\infty}} \ud x.
\end{equation}
Thanks to the asymptotic lower bound of moudulated energy $F_N(\mu, \bar{\rho})$ defined by \eqref{Quantity: modulated energy}, which has been proved by Serfaty in Corollary 3.5 in \cite{serfaty2020mean}, we can show that the equilibrium of particle system \eqref{Quantity: equilibrim of PS} is close to the equilibrium of the limiting equation \eqref{Quantity: limiting equilibrium} in the sense of relative entropy. This nice observation provides us with possibility to prove uniform-in-time propagation of chaos for particle systems with Coulombian and Risze interaction potentials. We summarize this observation as the following theorem.
\begin{theorem}\label{CRCon}(Concentration in relative entropy around the limiting equilibrium measure). Assume $d \geq 2$, $0 \leq s < d$ and $\|\bar{\rho}_0 \|_{L^{\infty}} < \infty$. The equilibrium of particle system and limiting equation are defined by \eqref{Quantity: equilibrim of PS} and \eqref{Quantity: limiting equilibrium} respectively. Then we have
\begin{equation*}
\mathcal{H}_N(\rho_{N,\infty} | \bar{\rho}_{N,\infty}) \leq \frac{\beta}{4d} \frac{\log N}{N} \textbf{1}_{s = 0} + \frac{C}{N^{1 - \frac{s}{d}}} \textbf{1}_{s \geq 0}.
\end{equation*}
where $\bar{\rho}_{N, \infty} = \bar{\rho}^{\otimes N}_{\infty}$ and $C = C(s, d, \beta,  \|\bar{\rho}_{\infty}\|_{L^{\infty}}, \|\bar{\rho}_0 \|_{L^{\infty}}) > 0$.
\end{theorem}
\begin{proof}
According to the definition of relative entropy \eqref{Quantity:Relative entropy}, we can compute
\begin{equation*}
\begin{aligned}
& \mathcal{H}_N(\rho_{N,\infty} | \bar{\rho}_{N,\infty}) + \mathcal{H}_N( \bar{\rho}_{N,\infty}|\rho_{N,\infty}) \\ = & - \frac{\beta}{2N^2} \int_{\R^{dN}} \rho_{N,\infty} \sum_{i \neq j} g(x_i - x_j) \ud X^N + \frac{\beta}{N} \int_{\R^{dN}} \rho_{N,\infty} \sum_{i=1}^N g \ast \rho(x_i) \ud X^N  \\ & + \frac{\beta}{2N^2} \int_{\R^{dN}} \bar{\rho}_{N,\infty} \sum_{i \neq j}^N g(x_i - x_j) \ud X^N  - \frac{\beta}{N} \int_{\R^{dN}} \bar{\rho}_{N, \infty} \sum_{i=1}^N g \ast \rho(x_i) \ud X^N \\ = & - \frac{\beta}{2} \int_{\R^{dN}} \rho_{N, \infty} \int_{\R^{2d}} g \ast \ud \mu_N(x) \ud \mu_N(x) \ud X^N + \frac{\beta}{N} \int_{\R^{dN}} \rho_{N,\infty} \sum_{i=1}^N g \ast \bar{\rho}_{\infty}(x_i) \ud X^N \\ & - \frac{\beta}{2} \int_{\R^{dN}} \rho_{N,\infty} \int_{\R^{2d}} g \ast \ud \bar{\rho}_{\infty}(x) \ud \bar{\rho}_{\infty}(x) \ud X^N + \frac{\beta}{2N} \Big(g(0) - \int_{\R^{2d}} g \ast \ud \bar{\rho}_{\infty}(x) \ud \bar{\rho}_{\infty}(x) \Big) \\ = & -\frac{\beta}{2} \int_{\R^{dN}} \rho_{N,\infty} \int_{\R^{2d}} g(x - y)(\ud \mu_N - \ud  \bar{\rho}_{\infty})^{\otimes 2}(x,y) \ud X^N - \frac{\beta}{2N} \int_{\R^{2d}} g \ast \ud \bar{\rho}_{\infty}(x) \ud \bar{\rho}_{\infty}(x).    \end{aligned}
\end{equation*}
Recall the asymptotic lower bound of modulated energy in Corollary 3.5 in \cite{serfaty2020mean},
\begin{equation*}
F_N(\mu_N, \bar{\rho}_{\infty}) \geq -\Big( \frac{1}{2d} \frac{\log(N \|\bar{\rho}_{\infty}\|_{L^{\infty}})}{N} \Big) \textbf{1}_{s=0} - C \|\bar{\rho}_{\infty}\|_{L^{\infty}}^{\frac{s}{d}} N^{-1 + \frac{s}{d}}\textbf{1}_{s \geq 0},
\end{equation*}
we immediately have 
\begin{equation*}
\begin{aligned}
& \mathcal{H}_N(\rho_{N,\infty} | \bar{\rho}_{N,\infty}) + \mathcal{H}_N( \bar{\rho}_{N,\infty}|\rho_{N,\infty}) \\ \leq & \frac{\beta}{4d} \frac{\log N}{N} \textbf{1}_{s = 0} + \frac{C}{N^{1 - \frac{s}{d}}} \textbf{1}_{s \geq 0} + \frac{C}{N}
\\ \leq & \frac{\beta}{4d} \frac{\log N}{N} \textbf{1}_{s = 0} + \frac{C}{N^{1 - \frac{s}{d}}} \textbf{1}_{s \geq 0},
\end{aligned}
\end{equation*}
where $C = C(s, d, \beta,  \|\bar{\rho}_{\infty}\|_{L^{\infty}}, \|\bar{\rho}_0 \|_{L^{\infty}})$. Here we use the nonincreasing of $L^{\infty}$ norm of limiting equation when we take replusive interacting force $K = -\nabla g$, i.e.
\begin{equation*}
\| \bar{\rho}(t, x) \|_{L^{\infty}} \leq \| \bar{\rho}_0 \|_{L^{\infty}} \ \ \ \text{for all} \ t \geq 0,
\end{equation*}
see detailed explaination in Section \ref{wellposedness} for $s = d - 2$ and Remark 3.9 in \cite{de2023sharp} for general $0 \leq s < d$. By nonnegativity of the relative entropy, we finish the proof.
\end{proof}

\begin{remark}\label{rekConvex}
 Similar arguments can be used to establish the same  concentration estimates of the Gibbs measure around the equilibirum of the mean-field equation in the sense of relative entropy under the assumptions that either  the potential $g$ is convex or $\beta \|g\|_{L^\infty}$ is small enough, given that $K = - \nabla g$. 
\end{remark}

\subsection{Related Literature}

Propagation of chaos and mean-field limit for the first order system given in the canonical form \eqref{Eq:particle system} have been rigorously derived for many models with different types of interacting kernels. The first attempt to establish such kind of propagation of chaos for the simplified physical model dates back to Kac \cite{kac1956foundations}, where partial results are obtained concerning the 1D spatially homogeneous Boltzmann equation. Later in the regime of classical coupling method, Dobrushin \cite{dobrushin1979vlasov} and Sznitman \cite{sznitman1991topics} give the quantitative convergence rate for Lipschitz kernel $K$ in the sense of Wasserstein-2 distance. Based on the synchronous coupling approach combined with the convexity of potentials, uniform-in-time propagation of chaos results are given by Malrieu \cite{Malrieu2001LogarithmicSI, Malrieu2003ConvergenceTE} using the logarithmic Sobolev inequality on the limit density. 

Recently, much progress has been extensively obtained in the mean-field limit for system \eqref{Eq:particle system} with singular interacting kernels. The \textit{relative entropy method} initiated in Jabin-Wang \cite{jabin2016mean,jabin2018quantitative} is an effective tool to obtain the strong propagation of chaos, by analyzing the evolution of relative entropy between the joint law of the particle system $\rho_N$ and the tensorized law of the limit system $\bar{\rho}_N$, controlling the time derivative with the relative entropy itself together with some $O(1/N)$ error term, bounded by the new Law of Large Numbers and Large Deviation theorems. We refer to \cite{jabin2018quantitative} for the 2D viscos vortex model on the torus case. Later extended results include Wynter \cite{wynter2021quantitative} for the so-called mixed sign model, Guillin-Le Bris-Monmarch\'e \cite{guillin2024uniform} for the uniform-in-time propagation of chaos using the logarithmic Sobolev inequality (LSI) on the limit density for making use of the negative relative Fisher quantity, Shao-Zhao \cite{shao2024quantitative} for the general circulation case with environmental noise, all of which are still working on the torus case. The recent preprint Feng-Wang \cite{feng2023quantitative} starts the program of extending the previous results to the whole Euclidean space $\R^2$, by deriving the logarithmic growth estimates for the limit density, say controlling terms like $\nabla \log \bar\rho$ and $\nabla^2 \log \bar\rho$ on the whole space. 

The \textit{modulated energy method} introduced by Serfaty has been viewed as another effective tool for resolving propagation of chaos results for singular kernel case, especially the Coulomb/Riesz case. By constructing a Coulomb/Riesz based metric between the empirical measure of the $N$-particle system and the limit density, Serfaty \cite{serfaty2020mean} shows the quantitative propagation of chaos for the first-order Coulomb and super-Coulomb cases without noises in the whole space, in the sense of controlling the weak convergence rate when acting on test functions. This method has been extended to various problems, including Rosenzweig \cite{rosenzweig2022mean} for 2D incompressible Euler equation, Nguyen-Rosenzweig-Serfaty \cite{nguyen2022mean} for Riesz-type flows, Rosenzweig-Serfaty \cite{rosenzweig2023global} for uniform-in-time estimates for sub-Coulomb cases, De Courcel-Rosenzweig-Serfaty \cite{de2023sharp} for periodic Riesz cases. The key estimates appearing in the control of modulated energy, say the commutator type estimates, have been thoroughly examined in the recent work Rosezweig-Serfaty \cite{rosenzweig2024sharp}. We refer to the recent lecture notes by Serfaty \cite{serfaty2024lectures} for more complete discussions on the above topics.

Our \textit{modulated free energy method}, which combines the relative entropy and the modulated free energy to treat more general singular kernels in the sense of so-called ``weighted relative entropy", is originated in \cite{bresch2019mean,bresch2019modulated,bresch2023mean}. Later De Courcel-Rosenzweig-Serfaty \cite{de2023attractive} has applied it for the comprehensive study of the attractive log gas. Generation of chaos results are also obtained partially by Rosenzweig-Serfaty \cite{rosenzweig2023modulated,rosenzweig2024relative} using the time-uniform LSI for the limit density or the modulated Gibbs measure. 

We also recommend some new directions of analyzing the BBGKY hiearachy solved by the marginals of the particle system, to obtain the local optimal convergence rate of relative entropy, while previous global-to-local arguments give suboptimal rates. This observation is initiated in Lacker \cite{lacker2021hierarchies} for bounded Lipschitz kernels and Lacker-Le Flem \cite{lacker2023sharp} for sharp uniform-in-time estimates. Recently Wang \cite{wang2024sharp} has extended to the 2D Navier-Stokes equation with high viscosity. Qualitative propagation of chaos results for 2D Vlasov-Poisson has also been established by Bresch-Jabin-Soler \cite{bresch2022new} using the BBGKY hierarchy and the compacteness argument. A new duality method based on the cumulants for weak propagation of chaos for general square-integrable kernels is introduced in the recent work Bresch-Duerinckx-Jabin \cite{breschduerjab}.

Now we come back to the system \eqref{Eq:particle system} with Coulombian interacting kernel on the whole space $\R^2$. Serfaty firstly derived the mean field limit for system \eqref{Eq:particle system} with Coulombian and Riesz potentials in \cite{serfaty2020mean} without diffusion by modulated energy, then Rosenzweig-Serfaty \cite{rosenzweig2023global} extended this method to the stochastic version to treat the system with diffusion in $\R^d$, but this result is limited to sub-Coulombian potentials. Recently, the combination of modulated energy and relative entropy successfully solves the propagation of chaos for system \eqref{Eq:particle system} with Coulombian potentials on $\mathbb{T}^d$ by \cite{bresch2019mean,bresch2019modulated,de2023sharp}, but the same results on the whole space $\R^d$ are still open. Also the uniform-in-time propagation of chaos of system \eqref{Eq:particle system} is restricted to the case of dimension $1$ on the whole space in \cite{rosenzweig2023modulated}. The key issue of failure of the modulated free energy on the wholes space is that $\bar{\rho}$ has no longer the positive lower bound, compared with the torus case, which is also the main difficulty for the extension to whole space for the relative entropy method. Moreover, the mass of solutions of the limit equation \eqref{Eq:limiting equation} will escape to infinity as $t \rightarrow \infty$, which causes the new difficulty to prove the uniform-in-time propagation of chaos. In this article, we firstly extend the modulated free energy to the whole space when $d = 2$ and derive the finite time propagation of chaos in the sense of relative entropy.

Finally, we give some comments of modulated free energy to end this subsection. Besides the form \eqref{Weighted relative entropy} of some \textit{weighted relative entropy}, the structure of modulated free energy has another inspiring form which appears in \cite{rosenzweig2023modulated}, say the relative entropy between the joint law and the Gibbs-type measure for the limit density,
\begin{equation*}
E_N(\rho_N|\bar{\rho}_N) = \frac{1}{\beta} \bigg( \mathcal{H}(\rho_N|\mathbb{Q}_{N,\beta}(\bar{\rho})) + \frac{\log K_{N,\beta}(\bar{\rho})}{N}\bigg)
\end{equation*}
where the new measure
\begin{equation*}
\mathbb{Q}_{N,\beta}(\bar{\rho}) = \frac{1}{K_{N,\beta}(\bar{\rho})} e^{-NF_N(\mu_N, \bar{\rho})} \ud \bar{\rho}_N
\end{equation*}
is called the \textit{modulated Gibbs measure}, and $K_{N,\beta}(\bar{\rho})$ is the associated partition function,
\begin{equation*}
K_{N,\beta}(\bar{\rho}) = \int_{\R^{dN}}e^{-NF_N(\mu_N, \bar{\rho})} \ud \bar{\rho}_N.
\end{equation*}
The authors in \cite{rosenzweig2023modulated} have proved that 
\begin{equation*}
|\log K_{N,\beta}(\bar{\rho})| = o(N),
\end{equation*}
so that we can understand the modulated free energy as the new relative entropy $\frac{1}{\beta} \mathcal{H}(\rho_N|\mathbb{Q}_{N,\beta}(\bar{\rho}))$. Also we have the associated new relative Fisher information
\begin{equation*}
I_N(\rho_N|\mathbb{Q}_{N,\beta}(\bar{\rho})) = \frac{1}{N} \int_{\R^{dN}} \left|\nabla  \sqrt{\frac{\rho_N}{\mathbb{Q}_{N,\beta}(\bar{\rho})}}\right|^2 \ud \mathbb{Q}_{N,\beta}(\bar{\rho}),
\end{equation*}
which is exactly the negative term appeared in the time evolution of modulated free energy (See Lemma \ref{Lemma:evolution of MFE}). It has been called the \textit{modulated Fisher information} in \cite{rosenzweig2023modulated}. In this point of view, the uniform functional inequalities for $modulated\ Gibbs\  measure$ $\mathbb{Q}_{N,\beta}(\bar{\rho})$ are very important in further study for uniform-in-time results.

\subsection{Structure of the Article} We organize the rest of this article as follows. In Section \ref{A priori regularity estimates}, we investigate some \textit{a priori} regularity results and long time asymptotic behavior of the limit equation. In Section \ref{Hamilton type gradient estimates}, we firstly prove the sharp Gaussian lower bound for the solution of Eq.\eqref{Eq:limiting equation} by a general maximum principle on the whole space, then we show the time decay estimates of logarithmic gradient and Hessian applying the sharp Gaussian lower bound and this general . In Section \ref{Proof of the main result}, we finish the proof of our main result Theorem \ref{Thm:EPoC}. In Appendix, we give the elementary calculations for the proof of two technical lemmas for the logarithmic growth estimates in Section \ref{Hamilton type gradient estimates}.  

\section{\textit{A Priori} Regularity Estimates}\label{A priori regularity estimates}

In this section, we present some basic regularity estimates about the limit equation \eqref{Eq:limiting equation}, which are useful in deriving our key logarithmic growth estimates and computing the entropy dissipation. Basically, we study the well-posedness arguments and the functional space that the solution belongs to. Also we apply the Carlen-Loss \cite{carlen1996optimal} method and the parabolic maximum principle to provide the long time decay of the solution, thanks to the special structure of the limit equation which is similar to that considered in \cite{carlen1996optimal} and improvements in \cite{rosenzweig2023global}. For the asymtotic long time decay of the Sobolev norms of the solution, we adapt the Kato \cite{kato1994navier} argument to conclude our desired results. We shall conclude by giving the Gaussian upper and lower bounds of the limit density, similar to that of the heat equation.

\subsection{Global Well-posedness}\label{wellposedness}

Consider the following Cauchy problem of the limit equation in the whole Euclidean space $\R^d$ with general dimension $d$,
\begin{equation}\label{Eq:Cauchy problem}
\left\{
\begin{aligned}
& \partial_t \bar{\rho} + \text{div}_x \big((K \ast \bar{\rho}) \bar{\rho}\big) = \frac{1}{\beta} \Delta \bar{\rho}, \\
& \bar{\rho}(\cdot,0) = \bar{\rho}_0.
\end{aligned}
\right.
\end{equation}
Then we can introduce the mild formulation of the Cauchy problem \eqref{Eq:Cauchy problem}. With $e^{t \Delta}$ denoting the classical heat flow in the whole space $\R^d$, we write that
\begin{equation*}\label{Eq:mild solution}
\bar\rho_t = e^{\frac{1}{\beta}t\Delta} \bar\rho_0 - \int_0^t
e^{\frac{1}{\beta}(t-s)\Delta}\text{div}(\bar\rho_s K \ast \bar\rho_s) \ud s.
\end{equation*}
Given any initial data $\bar{\rho}_0 \in L^{\infty}(\R^d) \cap L^1(\R^d)$, by a contraction mapping argument for the mild formulation, De Courcel-Rosenzweig-Serfaty \cite[Proposition 3.1, Proposition 3.2]{de2023sharp} have shown existence, uniqueness and continuous dependence on initial data in the Banach space $C([0,T], L^{\infty}(\R^d) \cap L^1(\R^d))$. Using the dependence on initial data estimates, we see that we can always approximate the solution to the Cauchy problem \eqref{Eq:Cauchy problem} by $C^{\infty}$ solution. Then we have $\bar{\rho}(x,t) \in C^{\infty}(\R^2 \times \R^+)$ and solves \eqref{Eq:Cauchy problem} in the classical sense. For the convinence of readers, let us give a brief explanation for strategy above in the following. 

As we have mentioned above, We do contraction mapping argument in Banach space $ X = C([0,T], L^1(\R^d) \cap L^{\infty}(\R^d))$. We decompose the interacting kernel $K$ as $K = K_1 + K_2$, where $K_1$ is a short-range interacting kernel supported on $B_r(0) \subset \R^d$ for some $r > 0$ and $K_2$ is a long-range smooth interacting kernel supported outside the $B_r(0)$. Now we use the fixed point argument by contraction map
$\mathcal{T}$ on $C([0,T], L^1(\R^d) \cap L^{\infty}(\R^d))$,
\begin{equation}\label{Eq:contraction map}
\mathcal{T}(\rho^{k+1}) = e^{\frac{1}{\beta} t \Delta} \rho_0 - \int_0^t
e^{\frac{1}{\beta}(t-s)\Delta}\text{div}(\rho^k_s K \ast \rho^k_s) \ud s,
\end{equation}
we aim to show that $\mathcal{T}$ is a contraction on the ball $B_R \subset X$ for $R, T > 0$ appropriately chosen. Assume that $\bar{\rho}_1$ and $\bar{\rho}_2$ are two solutions of Cauchy problem \eqref{Eq:Cauchy problem} with initial data $\bar{\rho}_0$, by the triangle inequality, for any $t \geq 0$,
\begin{equation*}
\| \mathcal{T}(\bar{\rho}_1) - \mathcal{T}(\bar{\rho}_2)\|_X \leq \bigg \| \int_0^t
e^{\frac{1}{\beta}(t-s)\Delta}\text{div}(\bar{\rho}_1 K \ast \bar{\rho}_1) \ud s - \int_0^t
e^{\frac{1}{\beta}(t-s)\Delta}\text{div}(\bar{\rho}_2 K \ast \bar{\rho}_2) \ud s \bigg \|_X. 
\end{equation*}
We may use Minkowski's inequality together with $\| e^{\frac{1}{\beta}(t-s)\Delta} \text{div}\|_{L^1} \lesssim (\frac{1}{\beta}(t-s))^{-\frac{1}{2}}$ to obtain
\begin{equation*}
\begin{aligned}
\| \mathcal{T}(\bar{\rho}_1) - \mathcal{T}(\bar{\rho}_2)\|_{L^{\infty}}  \leq & \bigg \| \int_0^t
e^{\frac{1}{\beta}(t-s)\Delta}\text{div}(\bar{\rho}_1 K \ast (\bar{\rho}_1 - \bar{\rho}_2)) \ud s \bigg \|_{L^{\infty}} \\ & + \bigg \| \int_0^t
e^{\frac{1}{\beta}(t-s)\Delta}\text{div}((\bar{\rho}_1 - \bar{\rho}_2) K \ast \bar{\rho}_2) \ud s \bigg \|_{L^{\infty}} \\ \leq & C \sqrt{\beta t} \| \bar{\rho}_1 - \bar{\rho}_2 \|_{L^{\infty} \cap L^1}(\| \bar{\rho}_1  \|_{L^{\infty}} + \| \bar{\rho}_2 \|_{L^{\infty}})
\end{aligned}
\end{equation*}
and
\begin{equation*}
\begin{aligned}
\| \mathcal{T}(\bar{\rho}_1) - \mathcal{T}(\bar{\rho}_2)\|_{L^1}  \leq & \bigg \| \int_0^t
e^{\frac{1}{\beta}(t-s)\Delta}\text{div}(\bar{\rho}_1 K \ast (\bar{\rho}_1 - \bar{\rho}_2)) \ud s \bigg \|_{L^1} \\ & + \bigg \| \int_0^t
e^{\frac{1}{\beta}(t-s)\Delta}\text{div}((\bar{\rho}_1 - \bar{\rho}_2) K \ast \bar{\rho}_2) \ud s \bigg \|_{L^1} \\ \leq & C \sqrt{\beta t} \| \bar{\rho}_1 - \bar{\rho}_2 \|_{L^{\infty} \cap L^1} \|\bar{\rho}_2 \|_{L^{\infty}}, 
\end{aligned}
\end{equation*}
here we use $K_1 \in L^1$ and $K_2 \in L^{\infty}$ and the constant $C = C(\|K_1\|_{L^1}, \|K_2\|_{L^{\infty}}, d)$. Combining these two inequality, we have
\begin{equation*}
\| \mathcal{T}(\bar{\rho}_1) - \mathcal{T}(\bar{\rho}_2)\|_{L^{\infty} \cap L^1} \leq C \sqrt{\beta t} \| \bar{\rho}_1 - \bar{\rho}_2 \|_{L^{\infty} \cap L^1}(\| \bar{\rho}_1  \|_{L^{\infty} \cap L^1} + \| \bar{\rho}_2 \|_{L^{\infty} \cap L^1}).
\end{equation*}
Now we take $R$ such that $2\|\bar{\rho}_0\|_{L^{\infty}} \leq R$ and do similar argument as above for $\bar{\rho}_i \in B_R$ in Banach space $X$, we have
\begin{equation*}
\begin{aligned}
\| \mathcal{T}(\bar{\rho}_i)\|_X \leq \frac{R}{2} + 2CR^2\sqrt{\beta T}, \ \ \ i = 1, 2. 
\end{aligned}
\end{equation*}
Choosing $T > 0$ such that 
\begin{equation*}
2CR \sqrt{\beta T} = \frac{1}{2},
\end{equation*}
we have $\mathcal{T}(\bar{\rho}_i) \in B_R$ in $X$ for $i = 1,2$, and
\begin{equation*}
\| \mathcal{T}(\bar{\rho}_1) - \mathcal{T}(\bar{\rho}_2)\|_X \leq \frac{1}{2} \| \bar{\rho}_1 - \bar{\rho}_2 \|_X, 
\end{equation*}
which shows that $\mathcal{T}$ is a contraction on $B_R$. Now we gain a covergence squence $\{ \bar{\rho}^k(x,t)\}_{k=1}^{\infty}$ in $X$ by interation \eqref{Eq:contraction map}. Taking $k \rightarrow \infty$, we obtain the limit $\bar{\rho}$ as the unique solution of \eqref{Eq:Cauchy problem} with $T \leq \frac{1}{C \beta \|\bar{\rho}_0\|_{L^{\infty}}}$.
Now we assume that $\bar{\rho}_0 \in C([0,T], W^{2,1}(\R^d) \cap W^{2, \infty}(\R^d))$, then we have
\begin{equation*}\label{Eq:first order representation}
\nabla \bar{\rho} = e^{\frac{1}{\beta} t \Delta} (\nabla \bar{\rho}_0) - \int_0^t e^{\frac{1}{\beta}(t-s) \Delta}\text{div}( \nabla \bar{\rho}_s K \ast \bar{\rho}_s + \bar{\rho}_s K \ast \nabla \bar{\rho}_s) \ud s,
\end{equation*}
by similar argument as above, we have
\begin{equation*}\label{IEq: first order estimates}
\| \nabla \bar{\rho}_t \|_{L^1 \cap L^{\infty}} \leq \| \nabla \bar{\rho}_0 \|_{L^1 \cap L^{\infty}} + C(K, \| \bar{\rho} \|_{L^1 \cap L^{\infty}}) \int_0^t (\frac{1}{\beta}(t-s))^{-\frac{1}{2}} \| \nabla \bar{\rho}_s \|_{L^1 \cap L^{\infty}} \ud s, 
\end{equation*}
then we use general Gronwall inequality by (Theorem 1, \cite{YE20071075}), we have $\nabla \bar{\rho}_t \in L^1 \cap L^{\infty}$ for each $t \in [0,T]$. Similarly, 
\begin{equation*}\label{IEq:Second order estimates}
\| \nabla^2 \bar{\rho}_t \|_{L^1 \cap L^{\infty}} \leq \| \nabla^2 \bar{\rho}_0 \|_{L^1 \cap L^{\infty}} + C_1(K, \bar{\rho}, \nabla \bar{\rho}, T) + C_2(K, \bar{\rho}, \nabla \bar{\rho}) \int_0^t (\frac{1}{\beta} (t-s))^{-\frac{1}{2}} \| \nabla \bar{\rho}_s \|_{L^1 \cap L^{\infty}} \ud s. 
\end{equation*}
Then we continue this procedure and do contraction mapping argument in Banach space $C([0,T], W^{k,1}(\R^d) \cap W^{k, \infty}(\R^d))$, we can obtain any order regularity of the unique solution $\bar{\rho} \in C([0,T], L^1(\R^d) \cap L^{\infty}(\R^d))$ but with smaller time interval $[0,T]$. In the following section, we only need the solution $\bar{\rho}(x,t) \in C([0,T], W^{2,1}(\R^d) \cap W^{2, \infty}(\R^d))$. 

Finally, we show that the solution we obtain above is, in fact, global, i.e. $T_{\text{max}} = + \infty$. This fact has been verified by nonincreasing of $\|\bar{\rho}(\cdot, t)\|_{L^{\infty}}$ in Remark 3.9 in \cite{de2023sharp}. Let us give a brief explaination. We select a constant $c \geq 0$ and use the fact that $g$ is the unique distribution solution to the equation $- \Delta g = c_d \delta_0$. By integration by part, we have  
\begin{equation}\label{evolution of infty norm}
\begin{aligned}
\frac{d}{dt} \int_{\R^d}(\bar{\rho}_t - c)_+ \ud x & =  \int_{\{ \bar{\rho}_t \geq c\}}(\text{div}_x \big(\bar{\rho}_t(\nabla g \ast \bar{\rho}_t) \big) + \frac{1}{\beta} \Delta \bar{\rho}_t) \ud x \\ & = \int_{\{ \bar{\rho}_t \geq c\}} \nabla \bar{\rho}_t \cdot \nabla g \ast \bar{\rho}_t \ud x - c_{d,s} \int_{\{ \bar{\rho}_t \geq c\}} \bar{\rho}^2_t \ud x - \frac{1}{\beta} \int_{\{ \bar{\rho}_t = c\}} \nabla \bar{\rho}_t \cdot \frac{\nabla \bar{\rho}_t}{|\nabla \bar{\rho}_t|} \ud x. 
\end{aligned}
\end{equation}
The last term is nonpositive and may be discarded. For the first term, observing that $\nabla \bar{\rho}_t 1_{\{ \bar{\rho}_t \geq c\}} = \nabla (\bar{\rho}_t - c)_+$ a.e., we use integrating by parts again 
\begin{equation*}
\int_{\{ \bar{\rho}_t \geq c\}} \nabla \bar{\rho}_t \cdot \nabla g \ast \bar{\rho}_t \ud x = \int_{\R^d} \nabla (\bar{\rho}_t - c)_+ \cdot \nabla g \ast \bar{\rho}_t \ud x = c_{d,s}\int_{\R^d} (\bar{\rho}_t - c)_+ \bar{\rho}_t \ud x
\end{equation*}
Similarly, $\bar{\rho}_t 1_{\bar{\rho}_t \geq c}= (\bar{\rho}_t - c)_+ + c1_{\bar{\rho}_t \geq c}$, which implies the right-hand side of \eqref{evolution of infty norm} is less than
\begin{equation*}
\begin{aligned}
- c_{d,s}c\int_{\{\bar{\rho}_t \geq c\}} \bar{\rho}_t \ud x.
\end{aligned}
\end{equation*}
Now we take the constant $c$ such that $c \geq \sup \bar{\rho}_0$, since the function
\begin{equation*}
t \rightarrow \int_{\R^d}(\bar{\rho}_t - c)_+ \ud x
\end{equation*}
is nonincreasing, we have
\begin{equation*}
\int_{\R^d}(\bar{\rho}_t - c)_+ \ud x \equiv 0 \ \ \ \text{for all} \ t > 0.
\end{equation*}
We have shown that $\sup \bar{\rho}_t \leq \sup \bar{\rho}_0$ for all $t \geq 0$ and we conclude that the solution $\bar{\rho}$ we obtain before is global well-posed.

\subsection{Long Time Asymptotic Decay}

We are now in the step of showing the long time asymptotic decay, or optimal smoothing, of the limit equation \eqref{Eq:Cauchy problem} given initial data as above. The optimal decay rates in time, as one can expect, satisfies the same temporal decay estimates as the linear heat equation. This kind of estimates have been proved by Rosenzweig-Serfaty \cite{rosenzweig2023global} for \eqref{Eq:Cauchy problem} with sub-Coulomb potentials. It is worth mentioning that they improve the classical Carlen-Loss argument, which requires the convolution-type coefficient in the divergence to be divergence-free, to the more general case. We follow the same strategy to obtain the similar result for \eqref{Eq:Cauchy problem} with Coulomb potentials when $d \geq 2$. An important ingredient for proving this result is the following sharp form of Gross's logarithmic Sobolev inequality (LSI), see for instance Gross\cite{Gross} or Carlen-Loss \cite{carlen1996optimal}.

\begin{proposition}[Gross's LSI]\label{lsi}
Let $a>0$. Then for all $f \in H^1(\R^d)$, we have
\begin{equation}
\int_{\R^d} |f(x)|^2 \log \Big(\frac{|f(x)|^2}{\|f \|^2_{L^2}} \Big) \ud x+\Big(d+\frac{d \log a}{2}\Big)\int_{\R^d} |f(x)|^2 \ud x \leq \frac{a}{\pi }\int_{\R^d} |\nabla f(x)|^2 \ud x.
\end{equation}
\end{proposition}

In our simplified setting, we assume that $\bar{\rho} \geq 0$ is a probability density function and $\bar{\rho} \in L^1(\R^2) \cap W^{2,\infty}(\R^2)$. We also assume on $\bar{\rho}_0$ the initial bounds \eqref{Ass:nabla-log}-\eqref{Ass:Gaussian-upper-bound}. These assumptions are consistent with those in our main result. In the following, we denote the vector field $\nabla g \ast \bar{\rho} = K \ast \bar{\rho}$ as $u(x,t)$.  The proof of the asymptotic decay of $L^p$ norms of $\bar{\rho}$ is essentially the same as \cite[Theorem 5]{carlen1996optimal}. The only difference is that \cite[Theorem 5, Condition (2.2)]{carlen1996optimal} can be replaced by the following condition
\begin{equation*}
\begin{aligned}
& \int_{\R^d} |u(x,t)|^q\Big(\nabla \cdot b(x,t)u(x,t)\Big) \ud x \geq 0,\ \  \text{for\ all}\ q \geq 0,\\
& \|b(\cdot,t)\|_{L^{\infty}} \le B(t) \ \ \ \ \text{and}\ \ \ \ \ \|c(\cdot,t)\|_{L^{\infty}} \leq C(t),
\end{aligned}
\end{equation*}
where $B(t)$ and $C(t)$ are some given continuous functions on $(0,\infty)$. By our notations in this paper, these conditions above read as
\begin{equation}
\begin{aligned}
\int_{\R^d} |\bar{\rho}(x,t)|^q\Big(\nabla \cdot u(x,t) \bar{\rho}(x,t)\Big) \ud x \geq 0,\ \  \text{for\ all}\ q \geq 0, \ \ \ \ \|u(\cdot,t)\|_{L^{\infty}} \le B(t), 
\end{aligned}
\end{equation}
which can exactly be verified. For the sake of completeness, we describe the complete proposition below without proof.
\begin{proposition}[Optimal Smoothing]\label{decaye1}
Suppose that $\bar{\rho} \in C([0,\infty);L^1(\R^d) \cap L^\infty(\R^d))$ is a solution to \eqref{Eq:Cauchy problem}. Let $1 \leq p \leq q \le \infty$, then for all $t>0$,
\begin{equation}
\|\bar{\rho}_t\|_{L^q} \leq \frac{K(q)}{K(p)}\bigg(\frac{4\pi t}{\beta\big(\frac{1}{p}-\frac{1}{q}\big)} \bigg)^{\frac{1}{q}-\frac{1}{p}}\|\bar{\rho}_0\|_{L^p},
\end{equation}
where
\begin{equation*}
K(q):=\frac{q'^{\frac{1}{q'}}}{q^{\frac{1}{q}}}, \qquad 1 \leq q \leq \infty.
\end{equation*}
\end{proposition}

We are now ready to derive the asymptotic decay of the $L^\infty$ norm of the limit equation \eqref{Eq:Cauchy problem}. 

\begin{lemma}[Asymptotic Decay]\label{alemma}
	Given any initial data $\bar{\rho}_0 \in L^{\infty}(\R^2)\cap L^1(\R^2)$, let $\bar\rho_t$ be the solution to the Cauchy problem \eqref{Eq:Cauchy problem}, one has that for any $t \in [0,\infty)$, $\bar{\rho}_t, u_t=K \ast \bar{\rho}_t \in L^{\infty}(\R^2)$ with 
    \begin{equation}
    \begin{aligned}
    & \|\bar{\rho}_t \|_{L^{\infty}} \leq \frac{\beta C}{1+t}, \\
	& \|u_t \|_{L^{\infty}} \le \frac{\sqrt{\beta} C}{\sqrt{1+t}}.
    \end{aligned}
	\end{equation}
\end{lemma}
\begin{proof}
    The fact that $\|\bar{\rho}_t\|_{L^{\infty}} \leq \|\bar{\rho}_0\|_{L^{\infty}}$ follows from the parabolic maximum principle and gives our result in the short time case. By the optimal smoothing Proposition \ref{decaye1} we have
    \begin{equation*}
    \|\bar{\rho}_t\|_{L^{\infty}} \leq \frac{\beta C}{t} \|\bar{\rho}_0\|_{L^1} \le \frac{\beta C}{t},
    \end{equation*}
    which gives our result in the long time case.
    For the $L^{\infty}$ estimates of $u_t$, by direct computation,
	\begin{equation*} 
    \begin{aligned}
    |K \ast \bar{\rho}_t(x)|
	& = \Big|\int_{\R^2} K(y) \bar{\rho}_t(x-y) \ud y \Big| \\
	& \leq \int_{|y| \leq R} |K(y)| \bar{\rho}_t(x-y) {\rm d}y + \int_{|y|> R} |K(y)| \bar{\rho}_t(x-y) \ud y \\
	& \leq R \|\bar{\rho}_t \|_{L^{\infty}}+\frac{1}{2\pi R} \|\bar{\rho}_t\|_{L^1}.
    \end{aligned}
	\end{equation*}
	Taking $R=1$ for instance, we deduce that $\|u_t \|_{L^{\infty}} \le C$, while taking $R=\sqrt{\frac{\|\bar{\rho}_t\|_{L^1}}{2 \pi \|\bar{\rho}_t\|_{L^\infty} }}$ as the optimal parameter, we get 
	\begin{equation*}
    \| K \ast \bar{\rho}_t \|_{L^{\infty}} \le \sqrt{\frac{1}{2\pi}} \|\bar{\rho}_t \|_{L^1}^{\frac{1}{2}} \|\bar{\rho}_t \|_{L^\infty}^{\frac{1}{2}} \le \sqrt{\frac{1}{2\pi}} \|\bar{\rho}_0 \|_{L^1}^{\frac{1}{2}} \|\bar{\rho}_0 \|_{L^\infty}^{\frac{1}{2}} \le \frac{\sqrt{\beta} C}{\sqrt{t}}.
	\end{equation*}
\end{proof}

Next we present another lemma to derive the long time decay rate of the $L^\infty$ norm of higher derivatives of the limit density, including time and spatial derivatives, in the sense of Kato's method.

\begin{lemma}[Asymptotic Decay for Higher Derivatives]\label{lemma2}
	Given the initial data $\bar{\rho}_0 \in \mathcal{S}(\R^2)$,  then the solution to the Cauchy problem \eqref{Eq:Cauchy problem} $\bar{\rho}_t \in \mathcal{S}(\R^2)$ and $u_t=K \ast \bar{\rho}_t \in \mathcal{S}(\R^2)$, with the asymptotic decay rate given by 
	\begin{equation} 
    \begin{aligned}
    &\|\partial_t^n \nabla^k \bar{\rho} \|_{L^{\infty}} \leq C(1+t)^{-1-k/2-n},  \\
	&\|\partial_t^n \nabla^k u \|_{L^{\infty}} \leq C(1+t)^{-1/2-k/2-n},
    \end{aligned}
	\end{equation}
	for any $t \in [0,\infty)$ and $n,k \in \mathbb{N}$.
\end{lemma}
\begin{proof}
	By applying the Gr\"onwall's argument for the weighed moment of higher order derivatives
	\begin{equation*} 
    \sum_{k=0}^n \int_{\R^2} (1+|x|^2)^m |\nabla^k \bar{\rho}_t(x)|^2 \ud x
	\end{equation*}
	(for arbitrarily large integers $m,n \in \mathbb{N}$) and using the Gagliardo-Nirenberg inequality
	\begin{equation*}
    \|\nabla^j f \|_{L^{\infty}} \leq C \|\nabla^m f\|_2^a \|f \|_2^{1-a} ,\qquad \frac{j}{2}+\frac{1}{2}=\frac{am}{2},
	\end{equation*}
	it is easy to see that if the initial data $\bar{\rho}_0 \in \mathcal{S}(\R^2)$, then it remains spatially Schwartz for $\bar{\rho}_t$. Using the limit evolution equation \eqref{Eq:Cauchy problem}, we can also prove that $\partial_t^n \bar{\rho}_t$ is also in $\mathcal{S}(\R^2)$ by a simple bootstrap argument. Thus given any $k \in N$, there exists a $t_k>0$ such that 
	\begin{equation*} 
    \|\partial_t^n \nabla^k \bar{\rho} \|_{L^{\infty}} \le C   
	\end{equation*}
	for $t \in [0,t_k]$, where $C$ depends on the initial data $\bar{\rho}_0$. And by $\nabla^k u=K \ast \nabla^k \bar{\rho}$ we have for $t \in [0,t_k]$,
    \begin{equation*}	
    \|\partial_t^n \nabla^k  u \|_{L^{\infty}} \leq C.
    \end{equation*} 
  For the long time decay, we note that the kernel $K(x)=\frac{1}{2\pi}\frac{x}{|x|^2}$ has the similar properties with $S(x)=\frac{1}{2\pi}\frac{(x_2,-x_1)}{|x|^2}$, so we can use the same method in \cite{kato1994navier} to deduce that
  \begin{equation*} 
  t^{n+\frac{k}{2}+1-\frac{1}{q}}\|\partial_t^n A^k \bar{\rho} \|_{L^q} \leq K \qquad q \in (1,\infty],
  \end{equation*}
  where $A=(-\Delta)^{\frac{1}{2}}$. Since we have proved that $\partial_t^n \bar{\rho}_t \in \mathcal{S}(\R^2)$, we have the following inequality
  \begin{equation*} 
  \|\partial_t^n \nabla^k  \bar{\rho}_t \|_{L^2} \le \|\partial_t^n A^k  \bar{\rho}_t \|_{L^2}
  \end{equation*}
  using the Gagliardo-Nirenberg inequaility again we deduce that
  \begin{equation*}
  \|\partial_t^n \nabla^k \bar{\rho}_t \|_{L^{\infty}} \leq c \|\partial_t^n \nabla^{2+2k} \bar{\rho}_t\|_{L^2}^{\frac{1}{2}} \|\bar{\rho}_t \|_{L^2}^{\frac{1}{2}} \leq  C \|\partial_t^n A^{2+2k} \bar{\rho}_t\|_{L^2}^{\frac{1}{2}} \|\bar{\rho} \|_{L^2}^{\frac{1}{2}} \le \frac{C}{t^{1+\frac{k}{2}+n}} 
  \end{equation*}
 And using the same method in Lemma \ref{alemma} and the bounds with  $\|\nabla^k \bar{\rho}_t \|_{L^{\infty}}$ we can deduce that
 \begin{equation*}
 \|\partial_t^n\nabla^k u_t \|_{L^{\infty}} \le  \frac{C}{t^{\frac{1}{2}+\frac{k}{2}+n}} .
 \end{equation*}
 So we complete the proof.
\end{proof}
 
Finally we present the Gaussian upper bound of the limit density, which is also in the sense of Carlen-Loss \cite{carlen1996optimal} argument for the upper bound estimate of the heat kernel type.
 
\begin{lemma}[Gaussian Upper Bound]\label{Pro: Gaussian upper bound}
Assume that the initial data $\bar{\rho}_0$ satisfies the Gaussian upper bound as in our main result,
	\begin{equation*} \bar{\rho}_0(x) \leq C_0 \exp(-C_0^{-1}|x|^2).
	\end{equation*}
Then we have the following Gaussian upper bound for $\bar{\rho}_t$ for any $t \in [0,T]$
	\begin{equation*} 
    \bar{\rho}(x,t) \le \frac{C}{1+t}\exp\Big(-\frac{|x|^2}{8t+C}\Big).
	\end{equation*}
\end{lemma}

\begin{proof}
The proof of Proposition \ref{decaye1} tells us that \cite[Theorem 5, Condition (2.2)]{carlen1996optimal} can be replaced by the following assumptions 
	\begin{equation*}
    \begin{aligned}
    &\int_{\R^n} |u(x,t)|^q\Big(\nabla \cdot b(x,t)u(x,t)\Big) \ud x \geq 0, \ \ \ \ \ \text{for all}\ \ q \geq 0,\\
	&\|b(\cdot,t)\|_{L^{\infty}} \le B(t), \qquad and \qquad \|c(\cdot,t)\|_{L^{\infty}} \le C(t),
    \end{aligned}
	\end{equation*}
	where $B(t)$ and $C(t)$ are given continuous functions on $(0,\infty)$. So the equality \cite[Theorem 3, (1.10)]{carlen1996optimal} still holds for our limit equation \eqref{Eq:Cauchy problem}, i.e.
	\begin{equation*} 
    \begin{aligned}
    \bar{\rho}(x,t) & \leq C \int_{\R^2} \frac{1}{t}\exp\Big(-\frac{|x-y|^2}{8t}\Big)\bar{\rho}_0(y) \ud y \\
	& \le  C \int_{\R^2} \frac{1}{t}\exp\Big(-\frac{|x-y|^2}{8t}-\frac{|y|^2}{C}\Big) \ud y \\
	& \le \int_{\R^2} \frac{1}{t}\exp\Big(-\frac{|x|^2}{8t+C}\Big)\exp\Big(-\frac{4t+C}{4Ct}|y|^2\Big) \ud y \\
	& \leq \frac{C}{1+t}\exp\Big(-\frac{|x|^2}{8t+C}\Big).
    \end{aligned}
	\end{equation*}
\end{proof}

\section{Logarithmic Growth Estimates}\label{Hamilton type gradient estimates}

In the previous work \cite{feng2023quantitative}, Feng and Wang derive some kind of logarithmic gradient and Hessian estimates  for the solution of 2D Navier-Stokes equation by a related maximum principle of diffusion operator developed in Grigor'yan \cite{grigor2006heat,grigoryan2009heat}. With the help of these estimates, they prove the propagation of chaos for 2D viscous vortex model on the whole space. In this section, we show the same maximum principle for our limit equation \eqref{Eq:limiting equation} with $\beta > 0$. Using this kind of maximum principle, we prove the logarithmic growth estimates for the solution of the limit equation \eqref{Eq:limiting equation} with $\beta > 0$.

\begin{theorem}[Logarithmic Gradient Estimate]\label{Thm:Logarithm gradient estimate}
Assume that the initial data $\bar{\rho}_0$ satisfies the following logarithmic growth conditions,
\begin{equation}\label{Thm:LGE-ass1}
\begin{aligned}
|\nabla \log \bar{\rho}_0(x)| & \lesssim 1 + |x|, \\ 
|\nabla^2 \log \bar{\rho}_0(x)| & \lesssim 1 + |x|^2,
\end{aligned}
\end{equation}
and the Gaussian upper bound that there exists some $C_0> 0$ such that
\begin{equation}\label{Thm:LGE-ass2}
\bar{\rho}_0(x) \leq C_0 \exp{(-C_0^{-1}|x|^2)}.
\end{equation}
Then we have the linear growth upper bound for the gradient of $\log \bar{\rho}$,
\begin{equation}\label{Thm:LGE-result}
|\nabla \log \bar{\rho}|^2 \leq \frac{C}{1+t}\Big(1 + \log(1+t) + \frac{|x|^2}{1+t}\Big)
\end{equation}
for some constant $C = C(\beta,C_0, \bar{\rho}) > 0$.
\end{theorem}

\begin{theorem}[Logarithmic Hessian Estimate]\label{Thm:Logarithm Hessian estimate} Assume that the initial data $\bar{\rho}_0$ satisfies the same conditions as in Theorem \ref{Thm:Logarithm gradient estimate}, then we have the quadratic growth upper bound for the Hessian of $\log \bar{\rho}$, 
\begin{equation}\label{Thm:LHE-result}
|\nabla^2 \log \bar{\rho}| \leq \frac{C}{1+t}\Big(1 + \log(1+t) + \frac{|x|^2}{1+t}\Big)
\end{equation}
for some constant $C = C(\beta, C_0, \bar{\rho}) > 0$.
\end{theorem}
\begin{remark}\label{Rmk:log-type estimates}
The same results are valid for 2D vortex model, 2D Patlak-Keller-Segel sytem and Coulombian system with higer dimension on the whole space for $t \in [0, T]$, once we verify the corresponding solution $\bar{\rho} \in ([0,T], W^{2,\infty}(\R^d))$.  
\end{remark}

\begin{remark}\label{Rmk: optimal log estimates}
The optimal time decay estimates of $|\nabla \log \bar{\rho}|$ and $|\nabla^2 \log \bar{\rho}|$ should be
\begin{equation*}
|\nabla \log \bar{\rho}|^2 \leq \frac{C}{1+t}\Big(1 + \frac{|x|^2}{1+t}\Big),
\end{equation*}
and
\begin{equation*}
|\nabla^2 \log \bar{\rho}| \leq \frac{C}{1+t}\Big(1 + \frac{|x|^2}{1+t}\Big).
\end{equation*}
It seems like that the method of maxmum principle cannot reach this optimal results. This suboptimal time decay results lead to suboptimal time growth esitmates in Theorem \ref{Thm:EPoC}.
\end{remark}

\subsection{Maximum Principle}

The main content of this subsection is the general version of maximum principle in \cite{grigor2006heat,grigoryan2009heat}
for the limit equation \eqref{Eq:limiting equation} on the weighted manifold. We denote $\Delta_f$ as the weighted Laplacian operator on $\R^d$ defined by $\Delta_f = \frac{1}{\beta}\Delta + \langle \nabla f, \nabla \rangle$, where $f = g \ast \bar{\rho}$. In fact, according to Section \ref{wellposedness}, given the initial data $\bar{\rho}_0 \in L^1$, then the solution $\bar{\rho}$ and the velocity field $u = \nabla g \ast \bar{\rho}$ are smooth. Hence the weighted Laplacian operator $\Delta_f$ has smooth coefficients.

Now let us describe the extended version of maximum principle in \cite{grigoryan2009heat} for the limit equation \eqref{Eq:limiting equation}, which is from \cite[Theorem 11.9]{grigoryan2009heat} for weighted heat equation and has been used to derive the propagation of chaos for 2D viscous vortex model on the whole space by Feng and Wang in \cite{feng2023quantitative}.

\begin{theorem}\label{Thm: Max-pri}
(Grigor'yan, an extended version). Let $(M,\tilde{g}, \ud \mu = e^f\ud V)$ be a complete weighted manifold, $q \in C
([0,T], W^{2,\infty}(\R^d))$, and let $F(x,t)$ be a solution of
\begin{equation}\label{Eq:weighted heat equation}
\partial_t F + qF \leq \Delta_f F \ \ \text{in} \ \ M \times (0,T], \ \ \ F(\cdot, 0) \leq 0.
\end{equation}
Assume that for some $x_0 \in M$ and for all $r > 0$,
\begin{equation}\label{MP:condition 1}
\int_0^T \int_{B(x_0,r)} F_{+}^2(x,t) \ud \mu \ud t \leq e^{\alpha(r)}
\end{equation}
foe some $\alpha(r)$ positive increasing function on $(0,\infty)$ such that
\begin{equation}\label{MP:condition 2}
\int_0^{\infty} \frac{r}{\alpha(r)} \ud r = \infty, 
\end{equation}
then $F \leq 0$ on $M \times [0,T]$.
\end{theorem}
\begin{proof}
We follow the same idea of \cite[Theorem 11.9]{grigoryan2009heat}. The key point is that the \textbf{Claim} in Theorem 11.9 remains true for the following equation (with nonnegative function $q$) $$\partial_t F +qF=\Delta_f F \ \ \text{in} \ \ M \times (0,T], \ \ \ F(\cdot, 0) = 0  $$

\begin{claim}
For any $R > 0$ satisfying the condition
\begin{equation*}
a - b \leq \frac{R^2}{8\alpha(4R)}, 
\end{equation*}
the following inequality holds:
\begin{equation*}
\int_{B_R} F^2(b, \cdot) \ud \mu \leq \int_{B_{4R}} F^2(a, \cdot) \ud \mu + \frac{C}{R^4}.
\end{equation*}
\end{claim}

\begin{proof}[Proof of \textbf{Claim 1}] 
We multiply the Equation $ \partial_t F + qF = \Delta_f F$ by test function $F \varphi^2 e^{\xi}$ ($\varphi$ and $\xi$ are exactly the same as defined in \cite{grigoryan2009heat}) and integrating it over $[b,a] \times M$, we obtain
\begin{equation*}
\begin{split}
\int F^2 \varphi^2 e^{\xi} \ud \mu & = \int_b^a \int \frac{\partial \xi}{\partial t} F^2 \varphi^2 e^{\xi} \ud \mu \ud t + 2 \int_b^a \int (\Delta_f F - qF) F \varphi^2 e^{\xi} \ud \mu \ud t \\ & \leq C \int_b^a \int |\nabla \varphi|^2 F^2 e^{\xi} \ud \mu \ud t,
\end{split}
\end{equation*}
then we finish the proof by the same argument in \cite[Theorem 11.9]{grigoryan2009heat}.
\end{proof}
By the Claim above and the same procedure of 
\cite[Theorem 11.9]{grigoryan2009heat}, we obtain that $F = 0 $ on $M \times (0,T]$ under the same conditions of Theorem 11.9 in \cite{grigoryan2009heat} for $q \in C([0,T], W^{2,\infty}(\R^d))$ and nonnegative, i.e.
\begin{equation*}
\partial_t F + qF = \Delta_f F \ \ \text{in} \ \ M \times (0,T], \ \ \ F(\cdot, 0) = 0 \ \ \ \Longrightarrow \ \ \ F = 0 \ \text{on} \ M \times (0,T].
\end{equation*}
For the general $q \in C
([0,T], W^{2,\infty}(\R^d))$, denote
$c=\sup \limits_{(x,t) \in M \times [0,T]} |q(x,t)|$. Let $F_1=e^{-ct}F$, then we have 
\begin{equation*}
\partial_t F_1+(q+c)F_1=\Delta_f F_1,\quad F(\cdot, 0) = 0.
\end{equation*}
Using the result for $q+c \in C
([0,T], W^{2,\infty}(\R^d))$ which is nonnegative, we know that $F_1 = 0 \ \text{on} \ M \times (0,T]$ and thus $F = 0 \ \text{on} \ M \times (0,T]$.
Moreover, since $\phi(x) = x_{+}$ is non-decreasing, convex, continuous and piecewise smooth, the conclusion above yields that
\begin{equation*}
\partial_t F_+ + qF_+ \leq \Delta_f F_+ \ \ \text{in} \ \ M \times (0,T], \ \ \ F_+(\cdot, 0) = 0 \ \ \ \Longrightarrow \ \ \ F_+ \leq 0 \ \text{on} \ M \times (0,T]
\end{equation*}
under the same conditions of Theorem \ref{Thm:Logarithm gradient estimate}. Hence we finish the proof.
\end{proof}

\subsection{Gaussian Lower Bound}

In this subsection, we apply the maximum principle above to provide the Gaussian lower bound for our limit equation \eqref{Eq:limiting equation}. We mention that compared to the Gaussian lower bound for 2D Navier-Stokes equation in the previous work \cite{feng2023quantitative}, which is obtained by classical Li-Yau estimates and parabolic Harnack inequality, our current estimate is optimal in the sense of the time dependence in the exponential, say $\exp(-|x|^2/t)$ instead of $\exp(-t|x|^2)$. This plays a key role for the improvement in our main result relative entropy estimate with respect to the growth rate of time $t$.

\begin{lemma}[Gaussian Lower Bound]\label{Pro: Gaussian lower bound}
Suppose that the initial data $\bar{\rho}_0$ satisfies the logarithmic gradient estimate as in our main result
\begin{equation*}
    |\nabla \log\bar\rho_0(x)| \lesssim 1+|x|.
\end{equation*}
Then we have the following Gaussian lower bound for the solution to the limit equation \eqref{Eq:Cauchy problem} for any time $t$,
\begin{equation*}
\bar{\rho}(t,x) \geq \frac{1}{C(1+t)^C} \exp\Big(-C\frac{|x|^2}{1+t}\Big)
\end{equation*}
for some $C> 0$.
\end{lemma}

\begin{proof}
We set two functions $\delta(t), \lambda(t) \in C^2([0, +\infty))$ to be confirmed and denote  
\begin{equation*}
\Gamma(t,x) = e^{\delta(t) + \lambda(t)|x|^2}
\end{equation*}
to be some Gaussian test function. We take the auxiliary function to be propagated as $F(t,x) = \Gamma(t,x) - \bar{\rho}(t,x)$. Now we compute directly that
\begin{equation*}
\begin{aligned}
(\partial_t + q - \Delta_f) \Gamma(t,x) = \big\{ \delta'(t) + \lambda'(t)|x|^2 + q - \frac{4}{\beta}\lambda(t) - \frac{4}{\beta} \lambda^2(t)|x|^2 + 2K \ast \bar{\rho} \cdot x \lambda(t) \big \} \Gamma(t,x).
\end{aligned}
\end{equation*}
Observing that $\Gamma(t,x) > 0$ and using the Young's inequality, we have
\begin{equation*}
\begin{aligned}
& (\partial_t + q - \Delta_f) \Gamma(t,x) \\ \leq & \Big\{ \delta'(t) + q - \frac{4}{\beta} \lambda(t) + C_{\varepsilon} \| K \ast \bar\rho \|^2_{L^{\infty}} + \Big(\lambda'(t) - \frac{4 - \varepsilon}{\beta} \lambda^2(t)\Big)|x|^2 \Big \} \Gamma(t,x) \\ \leq & \Big\{ \delta'(t) - \frac{4}{\beta} \lambda(t) + \|q\|_{L^{\infty}} + C_{\varepsilon} \| K \ast \bar{\rho} \|^2_{L^{\infty}} + \Big(\lambda'(t) - \frac{4 - \varepsilon}{\beta} \lambda^2(t)\Big)|x|^2 \Big \} \Gamma(t,x),  
\end{aligned}
\end{equation*}
for some small $\varepsilon > 0$.
Recalling we take $q = \bar{\rho}$ in our limit equation \eqref{Eq:limiting equation} and using the long time asymptotic behavior of $\bar{\rho}$ and $u = K \ast \bar{\rho}$ in Lemma \ref{alemma}, we have
\begin{equation*}
\begin{aligned}
(\partial_t + q - \Delta_f) \Gamma(t,x) \leq \Big\{ \delta'(t) - \frac{4}{\beta} \lambda(t) +\frac{C_\varepsilon}{1+t} + \Big(\lambda'(t) - \frac{4 - \varepsilon}{\beta} \lambda^2(t)\Big)|x|^2 \Big \} \Gamma(t,x). 
\end{aligned}
\end{equation*}
Now we firstly select the function $\lambda(t)$ to cancel the quadratic term $|x|^2$, say satisfying that 
\begin{equation*}
\lambda'(t) - \frac{4 - \varepsilon}{\beta} \lambda^2(t) = 0,
\end{equation*}
we have
\begin{equation*}
\lambda(t) = \frac{\lambda(0)}{1 - \frac{4 - \varepsilon}{\beta} \lambda(0) t}.
\end{equation*}
Next we choose $\delta(t)$ such that
\begin{equation*}
\begin{aligned}
\delta'(t) - 4 \lambda(t) + \frac{C_{\varepsilon}}{1+t} = 0
\end{aligned}
\end{equation*}
for some $C_{\varepsilon} > 0$. Then we have 
\begin{equation*}
\delta(t) = \delta(0) - C_{\varepsilon} - \frac{4\beta}{(4 - \varepsilon)} \log \Big(1 - \frac{4 - \varepsilon}{\beta} \lambda(0) t\Big) - C_{\varepsilon} \log (1+t).
\end{equation*}
In order to satisfy the initial condition for our maximum principle, i.e. $\Gamma(0,\cdot) \leq \bar\rho_0$, we take
\begin{equation*}
\lambda(0) = -C'', \ \ \ \delta(0) = -C'', 
\end{equation*}
hence we obtain 
\begin{equation*}
\begin{aligned}
\Gamma(t,x) = e^{-C_{\varepsilon}-C''} (1+t)^{-C_{\varepsilon}}\Big(1 + \frac{4 - \varepsilon}{\beta}C''t\Big)^{-\frac{4}{4 - \varepsilon}} \exp\Big(-\frac{C''|x|^2}{1 + C''\frac{4 - \varepsilon}{\beta}t}\Big).
\end{aligned}
\end{equation*}
Finally we conclude that
\begin{equation*}
(\partial_t + q - \Delta_f)(\Gamma(t,x) - \bar{\rho}) \leq 0 
\end{equation*}
and $\Gamma(0,\cdot) \leq \bar{\rho}_0$ by the above construction, and 
\begin{equation*}
\int_0^T \int_{B(x_0, r)}(\Gamma - \bar{\rho})^2 e^f \ud x \ud t \leq CTr^2
\end{equation*}
by the boundedness of $\Gamma(t,x), \bar{\rho}$ and $f$. Now we take $\alpha(r) = \log r$ and use the maximum principle Theorem \ref{Thm: Max-pri} to conclude that
\begin{equation*}
\begin{aligned}
\bar{\rho} \geq \Gamma(t,x) \geq e^{-C_{\varepsilon} -C''}\Big(1 + (4C'' + 1)t\Big)^{-2-C_{\varepsilon}} e^{-\frac{\beta|x|^2}{(4-\varepsilon) t+1/C''}},
\end{aligned}
\end{equation*}
for any small $\varepsilon < 2$, then we finish the proof.
\end{proof}

\subsection{Logarithmic Gradient Estimate}

Now we turn to the proof of our main logarithmic growth estimates. The basic idea is applying the Bernstein's method, say constructing some auxiliary functions containing some quantities to be estimated, then propagating the auxiliary functions along the parabolic operator and adapting the maximum principle. This will give us some bounds on the auxiliary functions, and thus on the desired quantities. In order to construct the auxiliary function $F(x,t)$ for our setting, we need to prepare some elementary calculations first.
\begin{lemma}\label{Lemma:LGE}
Assume that $\bar{\rho}$ solves the limit equation (\ref{Eq:limiting equation}) and recall that $u = -\nabla f = K \ast \bar{\rho}, q = div K \ast \bar{\rho} = \bar{\rho} \geq 0$. We also assume that $q \in C^1([0,T],W^{1,\infty}(\R^2))$. Then we have
\begin{equation}\label{Lemma:LGE-1}
\begin{split}
(\partial_t + q - \Delta_f )\Big(\frac{|\nabla \bar{\rho}|^2}{\bar{\rho}}\Big) & = - \frac{2}{\rho} \Big|\nabla_i \nabla_j \bar{\rho} - \frac{\nabla_i \bar{\rho} \nabla_j \bar{\rho}}{\bar{\rho}}\Big|^2 - \frac{2}{\bar{\rho}} \nabla \bar{\rho} \cdot \nabla (K \ast \bar{\rho}) \cdot \nabla \bar{\rho} - 2 \nabla \bar{\rho} \cdot \nabla q \\ & \leq \frac{C}{1+t} \frac{|\nabla \bar{\rho}|^2}{\bar{\rho}},
\end{split}
\end{equation}

\begin{equation}\label{Lemma:LGE-2}
\begin{split}
( \partial_t + q - \Delta_f )(\bar{\rho} \log \bar{\rho}) = - \frac{|\nabla \bar{\rho}|^2}{\bar{\rho}} - q \bar{\rho} \leq - \frac{|\nabla \bar{\rho}|^2}{\bar{\rho}}.
\end{split}
\end{equation}

\end{lemma}

The technical proof of the above lemma is based on direct calculations, which are similar to \cite[Lemma 4.1]{feng2023quantitative}, and we put them into the Appendix.

\begin{proof}[Proof of Theorem \ref{Thm:Logarithm gradient estimate}] 
For the short time result, we construct the auxiliary function to be propagated as
\begin{equation*}
\begin{aligned}
F(x,t) = \frac{|\nabla \bar{\rho}|^2}{\bar{\rho}} + C \bar{\rho} \log \bar{\rho} - C_1 \bar{\rho},
\end{aligned}
\end{equation*}
then we have
\begin{equation*}
\begin{aligned}
(\partial_t + q - \Delta_f )F(x,t) \leq C \frac{|\nabla \bar{\rho}|^2}{\bar{\rho}} - C \frac{|\nabla \bar{\rho}|^2}{\bar{\rho}} \leq 0.
\end{aligned}
\end{equation*}
Also, $F(x,0) \leq 0$ is equivalent to 
\begin{equation*}
|\nabla \log \bar{\rho}_0|^2 + C \log \bar{\rho}_0 \leq C_1,
\end{equation*}
which is valid thanks to the logathimic gradient estimate assumption \eqref{Thm:LGE-ass1} and the Gaussian upper bound assumption \eqref{Thm:LGE-ass2}. Finally, we verify the correction of the growth assumption \eqref{MP:condition 1}. Recall that $\nabla \bar{\rho} \in L^{\infty}$ and $\bar{\rho} \in L^{\infty}$ from Lemma \ref{alemma} and Lemma \ref{lemma2}, and the Gaussian lower bound Lemma \ref{Pro: Gaussian lower bound}, then we have
\begin{equation*}
\begin{aligned}
\int_0^T \int_{B_r} F_{+}^2(x,t)e^{f(x)}dxdt \leq T e^{C_T(1 + r^2)} \int_{B_r} e^{f(x)}dx \leq CTr^2 e^{C_T(1 + r^2)},
\end{aligned}
\end{equation*}
since $f = g \ast \bar{\rho}$ is bounded. Hence we may choose $\alpha(r) = C_Tr^2(1 + |\log r|)$ which satisfies \eqref{MP:condition 1} and \eqref{MP:condition 2}. Applying the maximum principle Theorem \ref{Thm: Max-pri}, we obtain that $F(x,t) \leq 0$, i.e.
\begin{equation}\label{maxpri-log}
|\nabla \log \bar{\rho}|^2 + C \log \bar{\rho} \leq C_1.
\end{equation}
Recall that the Gaussian lower bound Lemma \ref{Pro: Gaussian lower bound}
\begin{equation*}
\log \bar{\rho} \geq -C-C\log(1 + t)-C\frac{|x|^2}{1+t},
\end{equation*}
substituting this term into \eqref{maxpri-log}, we have
\begin{equation}\label{short time result}
|\nabla \log \bar{\rho}|^2 \leq C(1 + t)(1 + |x|^2),
\end{equation}
which gives us the short time result.

For the long time logarithmic gradient estimate for $\bar{\rho}$, we need to construct a new auxiliary function
\begin{equation*}
\tilde{F}(x,t) = \frac{t}{C+1}\frac{|\nabla \bar{\rho}|^2}{\bar{\rho}} + \bar{\rho} \log \bar{\rho} - C_1 \bar{\rho},
\end{equation*}
where the factor $t/(C+1)$ in $\tilde{F}$ is inspired by the proof of \cite[Theorem 4.1]{feng2023quantitative}. Similarly, we have
\begin{equation*}
(\partial_t + q - \Delta_f )\tilde{F}(x,t) \leq \frac{1}{C+1} \frac{|\nabla \bar{\rho}|^2}{\bar{\rho}} - \frac{|\nabla \bar{\rho}|^2}{\bar{\rho}} + \frac{t}{C+1} \frac{C}{1+t} \frac{|\nabla \bar{\rho}|^2}{\bar{\rho}} \leq 0.
\end{equation*}
By similar argument as the short time case, we can verify the growth assumptions \eqref{MP:condition 1} and \eqref{MP:condition 2} for $t \in [0, \infty)$ as before, and thus arrive at $\tilde{F} \leq 0$, i.e
\begin{equation}\label{LGE-nabla log rho}
|\nabla \log \bar{\rho}|^2 + \frac{C+1}{t} \log \bar{\rho} \leq \frac{C_1(C+1)}{t}.
\end{equation}
Substituting the lower bound of $\log \bar{\rho}$ in Proposition \ref{Pro: Gaussian lower bound},
\begin{equation*}
\log \bar{\rho} \geq -C-C\log(1+t)-C\frac{|x|^2}{1+t}
\end{equation*}
for some $C > 0$. We finally obtain
\begin{equation*}
|\nabla \log \bar{\rho}|^2 \leq \frac{C}{t}\Big(1+\log(1+t)+\frac{|x|^2}{1+t}\Big).
\end{equation*}
Combining with short time result \eqref{short time result}, we finish the logarithmic gradient estimate of $\bar{\rho}$ for $t \in [0, \infty)$.
\end{proof}

\subsection{Logarithmic Hessian Estimate}
Following the similar steps of the previous subsection, we further derive the logarithmic Hessian estimate for the limit density. Note that $\nabla^2 \log \bar{\rho} = \nabla^2 \bar{\rho}/\bar{\rho} - \nabla \log \bar{\rho} \otimes \nabla \log \bar{\rho}$, 
by our logarithmic gradient estimate Theorem \ref{Thm:Logarithm gradient estimate}, it suffices to control $\nabla^2 \bar{\rho}/\bar{\rho}$ by some quadratic function. We first need following calculations of propagation similar to Lemma \ref{Lemma:LGE}.

\begin{lemma}\label{Lemma:LHE}
Assume that $\bar{\rho}$ solves the limit equation (\ref{Eq:limiting equation}) and recall that $u = -\nabla f = K \ast \bar{\rho}, q = div K \ast \bar{\rho} = \bar{\rho} > 0$, we also assume that $q \in C^1([0,T],W^{1,\infty}(\R^2))$. Then we have
\begin{equation}\label{Lemma:LHE-1}
\begin{aligned}
(\partial_t + q - \Delta_f)\Big(\frac{|\nabla \bar{\rho}|^2}{\bar{\rho}}\Big) \leq - \frac{|\nabla^2 \bar{\rho}|^2}{\bar{\rho}} + \frac{2 |\nabla \bar{\rho}|^4}{\bar{\rho}^3} + \frac{C}{1+t}\frac{|\nabla \bar{\rho}|^2}{\bar{\rho}},
\end{aligned}
\end{equation}

\begin{equation}\label{Lemma:LHE-2}
\begin{aligned}
(\partial_t + q - \Delta_f)\Big(\frac{|\nabla^2 \bar{\rho}|^2}{\bar{\rho}}\Big) \leq \frac{C_1}{1+t
} \frac{|\nabla^2 \bar{\rho}|^2}{\bar{\rho}} + \frac{C_1'}{(1+t)^2} \frac{|\nabla \bar{\rho}|^2}{\bar{\rho}},
\end{aligned}
\end{equation}

\begin{equation}\label{Lemma:LHE-3}
\begin{aligned}
(\partial_t + q - \Delta_f)\Big(\bar{\rho} (\log \bar{\rho})^2\Big) = - \frac{2|\nabla \bar{\rho}|^2}{\bar{\rho}}(1 + \log \bar{\rho}) - 2 q \bar{\rho}  \log \bar{\rho}.
\end{aligned}
\end{equation}
\end{lemma}

The technical proof of Lemma \ref{Lemma:LHE} is based on direct calculations, which are similar to \cite[Lemma 4.2]{feng2023quantitative}, and we put them into the Appendix.

\begin{proof}[Proof of Theorem \ref{Thm:Logarithm Hessian estimate}] 
For the short time result, we construct the auxiliary function to be propagated as
\begin{equation*}
\begin{aligned}
F(x,t) = e^{-C_1t} \frac{|\nabla^2 \bar{\rho}|^2}{\bar{\rho}} -C \bar{\rho}(\log \bar{\rho})^2 + C_2
 \bar{\rho} \log \bar{\rho} -C'(1+t) \bar{\rho}.
\end{aligned}
\end{equation*}
Recall that $\bar{\rho} \in L^{\infty}$ and the Gaussian upper bound Lemma \ref{Pro: Gaussian upper bound}, then we can assume there exists some $A > 0$ such that $\log \bar{\rho} \leq A$ holds for all time $t$. Hence we select $C_2, C'$ large enough, i.e. $C_1' + 2C(1 + A) < C_2$ and $C'>2CAe^A$, and we have 
\begin{equation*}
\begin{aligned}
(\partial_t + q - \Delta_f)F(x,t) \leq & -C_1 e^{-C_1 t}\frac{|\nabla^2 \bar{\rho}|^2}{\bar{\rho}} + C_1 e^{-C_1 t} \frac{|\nabla^2 \bar{\rho}|^2}{\bar{\rho}} + C_1' e^{-C_1 t} \frac{|\nabla \bar{\rho}|^2}{\bar{\rho}} \\ & + 2C(1 + \log \bar{\rho}) \frac{|\nabla \bar{\rho}|^2}{\bar{\rho}} + 2C q \bar{\rho} \log \bar{\rho} - C_2 \frac{|\nabla \bar{\rho}|^2}{\bar{\rho}}-C'\bar\rho \\ \leq & \,\Big(C_1' e^{-C_1t} + 2C(1 + \log \bar{\rho}) - C_2\Big) \frac{|\nabla \bar{\rho}|^2}{\bar{\rho}} + 2C \bar{\rho}^2 \log \bar{\rho}-C'\bar\rho \leq 0.
\end{aligned}
\end{equation*}
Also, $F(x,0) \leq 0$ is equivalent to 
\begin{equation*}
\begin{aligned}
\frac{|\nabla^2 \bar{\rho}_0|^2}{\bar{\rho}_0^2} + C_2 \log \bar{\rho}_0 \leq C \bar{\rho_0}(\log \bar{\rho_0})^2 + C'. 
\end{aligned}
\end{equation*}
By our initial logarithmic Hessian estimate assumption \eqref{Thm:LGE-ass1}, the left-handside is bounded by some $C_3 + C_3|x|^2$ from above. By our initial Gaussian upper bound and logarithmic gradient esitmate assumptions \eqref{Thm:LGE-ass1} and \eqref{Thm:LGE-ass2}, the right-handside is bounded by some $C_4 + C_4|x|^2$ from below for $|x|$ large, and hence for all $x$. Now we can select $C_4$ sufficiently large by choosing  $C, C'$ large enough. Hence the initial condition is satisfied. Finally we verify the growth conditions \eqref{MP:condition 1} and \eqref{MP:condition 2} of Theorem \ref{Thm: Max-pri}. Recall that $\nabla^2 \bar{\rho} \in L^{\infty}$, $\nabla \bar{\rho} \in L^{\infty}$ and $ \bar{\rho} \in L^{\infty}$ from the long time asymptotic decay Lemma \ref{alemma} and Lemma \ref{lemma2}, and the Gaussian lower bound Lemma \ref{Pro: Gaussian lower bound}, we have
\begin{equation*}
\begin{aligned}
\int_0^T \int_{B_r} F_{+}^2(x,t)e^{f(x)}dxdt \leq T e^{C(1 + T)(1 + r^2)} \int_{B_r} e^{f(x)}dx \leq CTr^2 e^{C_T(1 + r^2)},
\end{aligned}
\end{equation*}
since $f = g \ast \bar{\rho}$ is bounded. Hence we may choose $\alpha(r) = C_Tr^2(1 + |\log r|)$ which satisfies \eqref{MP:condition 1} and \eqref{MP:condition 2}. Applying the maximum principle Theorem \ref{Thm: Max-pri}, we obtain that $F(x,t) \leq 0$, i.e.
\begin{equation}\label{short time result-LHE1}
\frac{|\nabla^2 \bar{\rho}|^2}{\bar{\rho}^2} \leq e^{C_1t}\Big(C \bar{\rho}(\log \bar{\rho})^2 - C_2 \log \bar{\rho} + C'(1+t)\Big).
\end{equation}
Recall that the Gaussian lower bound Lemma \ref{Pro: Gaussian lower bound}
\begin{equation*}
\log \bar{\rho} \geq -C-C\log(1 + t)-C\frac{|x|^2}{1+t} \ \text{and} \ \log \bar{\rho} \leq A, 
\end{equation*}
substituting this term into \eqref{short time result-LHE1}, we have
\begin{equation}\label{short time result-LHE}
\frac{|\nabla^2 \bar{\rho}|^2}{\bar{\rho}^2} \leq e^{C_1t}(1 + t)(1 + |x|^4),
\end{equation}
which gives us the short time result.

For the long time logarithmic Hessian estimate for $\bar{\rho}$, we construct a new auxiliary function as
\begin{equation*}
\tilde{F}(x,t) = \phi \frac{|\nabla^2 \bar{\rho}|^2}{\bar{\rho}} + \varphi \frac{|\nabla \bar{\rho}|^2}{\bar{\rho}} - A_1 \bar{\rho}(\log \bar{\rho})^2 + A_2 \bar{\rho} \log \bar{\rho} - A_3 \bar{\rho}.
\end{equation*}
where $A_1, A_2, A_3 > 0$. The two factors $\phi, \varphi$ in $\tilde{F}$ are inspired by the proof of \cite[Theorem 4.4]{feng2023quantitative}, which are defined as
\begin{equation*}
\phi(t) = \frac{t^2}{(C+1)(C_1+2)}, \ \ \ \varphi(t) = \frac{A_1 t}{C+1}.
\end{equation*}
Then we deduce that 
\begin{equation}\label{tilde-F}
\begin{aligned}
(\partial_t - q - \Delta_f)\tilde{F}(x,t) \leq & \Big(\phi' + \frac{C_1}{t} \phi - \varphi\Big) \frac{|\nabla^2 \bar{\rho}|^2}{\bar{\rho}} \\ & + \Big(\frac{C_1'}{t^2} \phi + \varphi' + \frac{C\varphi}{t} + 2\phi|\nabla \log \bar{\rho}|^2 + 2A_1(1 + \log \bar{\rho}) - A_2\Big) \frac{|\nabla \bar{\rho}|^2}{\bar{\rho}}.
\end{aligned}
\end{equation}
Note that
\begin{equation*}
\phi' + \frac{C_1}{t} \phi - \varphi \leq \bigg[ \frac{2}{(C+1)(C_1 + 1)} + \frac{1}{C+1} - \frac{A_1}{C+1} \bigg]t,
\end{equation*}
then we can choose $A_1$ large enough to make sure that the first term of the right-handside of \eqref{tilde-F} is non-positive. In terms of the second term of the right-handside of \eqref{tilde-F}, we use \eqref{LGE-nabla log rho} in the proof of Theorem \ref{Thm:Logarithm gradient estimate}, 
\begin{equation*}
2 \phi|\nabla \log \bar{\rho}|^2 \leq 2A_1 C_1 - 2 A_1 \log \bar{\rho}.  
\end{equation*}
Then we have
\begin{equation*}
\begin{aligned}
\frac{C_1'}{t^2} \phi + \varphi' + \frac{C\varphi}{t} + 2\phi|\nabla \log \bar{\rho}|^2 + 2A_1(1 + \log \bar{\rho}) - A_2 \leq \frac{1}{C+1} + A_1 + 2A_1(C_1 + 1) - A_2.
\end{aligned}
\end{equation*}
Now we can choose $A_2$ large enough to make sure the the second term of the right-handside of \eqref{tilde-F} is non-positive. By similar argument as the short time case, we can verify the growth assumptions \eqref{MP:condition 1} and \eqref{MP:condition 2} for $t \in [0, \infty)$ as before, and thus arrive at $\tilde{F} \leq 0$, i.e.
\begin{equation}\label{LHE-nabla log rho}
t^2 \frac{|\nabla^2 \bar{\rho}|^2}{\bar{\rho}^2} \leq C(\log \bar{\rho})^2 + C
\end{equation}
Substituting the lower and upper bound of $\log \bar{\rho}$ in Lemma \ref{Pro: Gaussian lower bound} and Lemma \ref{Pro: Gaussian upper bound}
\begin{equation*}
\log \bar{\rho} \geq -C-C\log(1+t) - C\frac{|x|^2}{1+t} \ \ \ \text{and} \ \ \ \log \bar{\rho} \leq C-\log(1+t)-\frac{|x|^2}{C(1+t)},
\end{equation*}
we finally obtain that
\begin{equation*}
\frac{|\nabla^2 \bar{\rho}|^2}{\bar{\rho}^2} \leq \frac{C}{t^2}\Big(1 + \log^2(1+t) + \frac{|x|^4}{t^2}\Big).
\end{equation*}
Combining with the short time result \eqref{short time result-LHE}, we finish the logarithmic Hessian estimate of $\bar{\rho}$ for $t \in [0, \infty)$.
\end{proof}

\section{Proof of Main Result}\label{Proof of the main result}

Now we are arriving at giving the proof of our main result. The main idea, as explained before, is to extend the modulated free energy method, which has been successfully applied on the torus case to derive the propagation of chaos results for 2D repulsive Coulomb kernel and mild attractive kernel (including the 2D Patlak-Keller-Segel kernel) in the series of works \cite{bresch2019mean,bresch2019modulated,bresch2023mean}, to the whole Euclidean space. Recall the definition of modulated free energy on the whole space between the joint law $\rho_N$ and the tensorized law $\bar\rho_N$ given by the following weighted relative entropy form
\begin{equation}
E_N(\rho_N|\bar{\rho}_N) = \frac{1}{N} \int_{\R^{dN}} \rho_N \log \bigg( \frac{\rho_N(t, X^N)}{G_N(t, X^N)} \frac{G_{\bar{\rho}_N}(t, X^N)}{\bar{\rho}_N(t, X^N
)}\bigg) \ud X^N.
\end{equation}
We aim at controlling the time derivative of the modulated free energy by itself, together with some extra small error terms. Then the desired result comes from the classical Gr\"onwall argument.

\subsection{Time Evolution of Modulated Free Energy}

In this subsection we present the time derivative of the modulated free energy on the whole space, which is almost the same as the torus case for Coulomb potentials. Therefore, we adapt the proof of \cite[Proposition 2.1]{bresch2019modulated} and claim the following lemma, and refer the proof to \cite{bresch2019modulated}.
\begin{lemma}[Time Evolution of Modulated Free Energy]\label{Lemma:evolution of MFE}
Assume that $\rho_N$ is an entropy solution to the Liouville equation \eqref{Eq:Liouville} and assume that $\bar{\rho} \in C([0,\infty), W^{2,1} \cap W^{2,\infty}(\R^2))$ solves the limit equation \eqref{Eq:limiting equation}. Then the time evolution of modulated free energy satisfies that
\begin{equation}\label{IEq:time derivative of MFE}
\begin{aligned}
E_N(\rho_N|\bar{\rho}_N) \leq & \,E_N(\rho_N|\bar{\rho}_N)(0) - \frac{1}{\beta N} \int_0^t \int_{\R^{2N}} \bigg| \nabla \log \frac{\rho_N}{\bar{\rho}_N} - \nabla \log \frac{G_N}{G_{\bar{\rho}_N}}\bigg|^2 \ud \rho_N(x)
\\ & - \frac{1}{2} \int_0^t \int_{\R^{2N}} \int_{(\R^2)^2 \backslash \Delta} \Big(w(x) - w(y)\Big) \cdot \nabla g(x-y) ( \ud \mu_N - \ud \bar{\rho})^{\otimes 2} \ud \rho_N(X^N), 
\end{aligned}
\end{equation}
where $\mu_N = \frac{1}{N} \sum_{i=1}^N \delta_{x_i}$ is the empirical measure of the particle system \eqref{Eq:particle system} and the vector field in the commutator term is given by $w(x) = \nabla \log \bar{\rho}(x) + \beta \nabla g \ast \bar{\rho}(x)$.
\end{lemma}
\begin{remark}
Note that
\begin{equation*}
w(x) = \nabla \log \frac{\bar{\rho}}{G_{\bar{\rho}}}(x) = \nabla \log \bar{\rho}(x) + \beta \nabla g \ast \bar{\rho}(x). 
\end{equation*}
As pointed out in previous works \cite{bresch2019modulated}, the weighted combination of relative entropy and modulated energy, which forms the modulated free energy, exactly cancels the term $\text{div} K$ appearing in the evolution of relative energy. In the Coulomb potential case we consider here, we have $\udiv K=\Delta g=\delta_0$, hence any single argument using only relative entropy $\mathcal{H}_N(\rho_N|\bar{\rho}_N)$ or modulated energy $F(\mu_N, \bar{\rho})$ individually will not work. Jabin-Wang \cite{jabin2018quantitative} uses the relative entropy to deal with the 2D Boit-Savart law kernel case thanks to that $\text{div} K = 0$, while Rosenzweig-Serfaty \cite{rosenzweig2023global} resolves the propagation of chaos of sub-Coulomb potentials in the whole space because of the available estimates on $\Delta g$ when $g$ is in the sub-Coulomb range.
\end{remark}

\subsection{Large Deviation Theorem}

In the original paper \cite{jabin2018quantitative} for entropic propagation of chaos for 2D Navier Stokes equation, Jabin-Wang developed some Large Deviation type theorem, in order to control the exponential integral of some large sum of test functions with respect to the tensor of i.i.d. laws $\bar\rho^{\otimes N}$, which naturally appears during the computation of the time evolution of the normalized relative entropy. By the Taylor's expansion of the exponential function and counting for each term carefully that how many terms in the summation actually vanish, thanks to the two-side cancellation conditions of the test function \eqref{LDP:condition1}, one is able to bound the large integral by some $O(1/N)$ quantity, coinciding with the global optimal rate for mean-field convergence, while one may face $O(1)$ bound by crucial estimates. Now that our commutator term appearing in the last term of \eqref{IEq:time derivative of MFE} also has a natural form of such exponential integral, we wish to control it by the Large Deviation Theorem, i.e. \cite[Theorem 4]{jabin2018quantitative}. We state the result in this subsection while referring the proof to \cite{jabin2018quantitative}.

\begin{theorem}[Large Deviation]\label{Thm:LDP}
Consider any $\phi(x,y)$ satisfying the canceling properties 
\begin{equation}\label{LDP:condition1}
\int \phi(x,y) \ud \bar{\rho}(x) = 0\ \ \text{for all}\ x, \ \ \ \int \phi(x,y) \ud \bar{\rho}(y) = 0\ \ \text{for all}\ y,
\end{equation}
and
\begin{equation}\label{LDP:condition2}
\gamma = C \bigg( \sup_{p \geq 1} \frac{\|\sup_y |\phi(\cdot,y)|\|_{L^p(\bar{\rho} \ud x)}}{p}\bigg) < 1,
\end{equation}
Then we have 
\begin{equation*}
\int_{\R^{2N}} \exp \bigg( \frac{1}{N} \sum_{i,j=1} \phi(x_i,x_j) \bigg) \ud \bar{\rho}_N \leq \frac{2}{1- \gamma} < \infty, 
\end{equation*}
where $\bar{\rho} = \bar{\rho}^{\otimes N}$.
\end{theorem}

We mention that the condition \eqref{LDP:condition2} for $\gamma$ being small is automatically satisfied once the test function $\phi \in L^\infty$ with small enough $L^\infty$ norm, which is in its nature due to the convergence radius for some Taylor's series, see the combinatory proof of \cite{jabin2018quantitative}. Later Lim-Lu-Nolen \cite{lim2020quantitative} also provides a simpler probabilistic proof using the martingale expansion and martingale inequalities, for the special case that $\phi \in L^\infty$.

However, the simple case stated above can only help us deal with the torus case, as there exists some $\nabla \log \bar\rho$ terms in the explicit form of the test function $\phi$, which is only able to be bounded on the compact domain once we make the assumption that
\begin{equation*}
    \inf \bar\rho \geq \lambda>0
\end{equation*}
as in the series of previous works \cite{jabin2018quantitative,wynter2021quantitative,guillin2024uniform,shao2024quantitative}. For our whole space case, we can never expect such a lower bound on $\bar\rho$ and hence never get $\nabla \log \bar\rho \in L^\infty$. But the good news is, as observed firstly in \cite{feng2023quantitative}, the condition \eqref{LDP:condition2} allows the test function $\phi$ to be only bounded in $y$ with quadratic growth in $x$, once we have the Gaussian upper bound for the limit density $\bar\rho$. This relaxation of requirements meets with our case exactly.

\subsection{Control of Modulated Free Energy}

Now we proceed to estimate the time derivative of modulated free energy given in the first subsection. Note that the second term of the right-handside of \eqref{IEq:time derivative of MFE} is nothing but the minus time integral of some normalized modulated Fisher information term, which is obviously non-positive. Hence we are able to neglect it directly. From now on we only focus on the third error term of the right-handside of \eqref{IEq:time derivative of MFE}. Since $K(x) = \nabla g(x)$ is odd, i.e. $K(-x)=-K(x)$, we rewrite this term by the classical symmetrization trick into
\begin{equation}\label{IEq: functional inequality}
\begin{aligned}
& \int_{\R^{2N}} \int_{(\R^2)^{2N}\backslash \Delta} (w(x) - w(y)) \cdot K(x-y) ( \ud \mu_N - \ud \bar{\rho})^{\otimes 2} \ud \rho_N = \int_{\R^{2N}} \bigg( \frac{1}{N^2} \sum_{i,j=1}^N \phi(x_i, x_j) \bigg) \ud \rho_N,
\end{aligned}
\end{equation}
where the symmetric test function $\phi(x,y)$ is given by
\begin{equation}\label{LDP:vector field}
\begin{split}
\phi(x,y) =\varphi(x,y)-\int \varphi(z,y) \bar\rho(z) \ud z-\int \varphi(x,w) \bar\rho(w) \ud w+\int\int \varphi(z,w) \bar\rho(z) \bar\rho(w) \ud z \ud w.
\end{split}
\end{equation}
Here $\varphi(x,y)=(w(x)-w(y)) \cdot K(x-y)$. This expression can be viewed as some kind of transformation of the test function $\varphi(x,y)$ into $\phi(x,y)$, making it satisfy the two-side cancellation conditions \eqref{Eq:Cancellation inequality}. We denote that the singular part and the regular part respectively by
\begin{equation*}\label{w_1, w_2}
w_1(x) = \int w(y) \cdot K(x-y) \bar{\rho}(y) \ud y, \ \ \ w_2(x) = w(x) \cdot K \ast \bar{\rho}(x),
\end{equation*}
and the constant term
\begin{equation}\label{constant c}
c = \int w(z) \cdot K \ast \bar{\rho}(z) \ud \bar{\rho}(z).
\end{equation}

\begin{remark}
For condition $\nabla \cdot K = 0$ and $w(z) = \nabla \log(z)$, we have $c = 0$ and $w_1(z) = 0$, which is actually the case of Biot-Savart law in \cite{feng2023quantitative}.
\end{remark}
We also denote the sum of test functions on the left-handside of \eqref{IEq: functional inequality} as 
\begin{equation*}
\Phi_N(x_1,...,x_N) = \frac{1}{N^2} \sum_{i,j=1}^N \phi(x_i,x_j), 
\end{equation*}
then we use the Gibbs inequality with some positive function $\eta(t)$, which has been proved in \cite[Lemma 1]{jabin2018quantitative}, to obtain 
\begin{equation}\label{LDP:wait-eta}
\begin{aligned}
E_N(\rho_N|\bar{\rho}_N) \leq &\, E_N(\rho_N|\bar{\rho}_N)(0)
- \frac{1}{2} \int_0^t \int_{\R^{2N}} \bigg( \frac{1}{N^2} \sum_{i,j=1}^N \phi(x_i, x_j) \bigg) \ud \rho_N \ud s\\ \leq &\, E_N(\rho_N|\bar{\rho}_N)(0) + \int_0^t \frac{1}{\eta(s)} \mathcal{H}_N(\rho_N|\bar{\rho}_N)(s) \ud s + \frac{1}{N} \int_0^t \frac{1}{\eta(s)} \log \int_{\R^{2N}} \bar{\rho}_N \exp(N \eta \Phi) \ud X^N \ud s.
\end{aligned}
\end{equation}
Now it suffices to bound the exponential integral
\begin{equation*}
\int_{\R^{2N}} \bar{\rho}_N \exp(N \eta \Phi) \ud X^N,
\end{equation*}
which is integrated with respect to tensor product of i.i.d. probability density functions. Now we can apply our Large Deviation Theorem \ref{Thm:LDP}. It suffices to verify that the symmetric test function $\phi(x,y)$ defined in \eqref{LDP:vector field} satisfies the two-side cancellation conditions \eqref{LDP:condition1} and the upper bound \eqref{LDP:condition2} if we choose suitable weight $\eta(t) > 0$.

As explained in the previous subsection, when we consider the singular interacting kernel $K$ which has singularity equivalent to $1/|x|$ near the origin, for instance the Coulomb case of $d = 2$ in this paper, the test function $\phi(x,y)$ will be bounded in $y$ and have at most quadratic growth in $x$ on the whole space, so that we can control the functional \eqref{IEq: functional inequality} by the relative entropy $\mathcal{H}_N(\rho_N|\bar{\rho}_N)$ itself and a small error term of $O(1/N)$. This idea has been used for the Biot-Savart law on $\R^2$ in \cite{feng2023quantitative} and the Patlak-Keller-Segel model on $\mathbb{T}^2$ in \cite{bresch2023mean}. Now we detail this critical observation for our 2D Coulomb case.

\begin{lemma}
The constant $c$ is finite and the test function $\phi(x,y)$ satisfies the two-side cancellation conditions, i.e.
\begin{equation*}\label{Eq:Cancellation inequality}
\int \phi(x,y) \ud \bar{\rho}(x) = 0, \ \ \ \int \phi(x,y) \ud \bar{\rho}(y) = 0.
\end{equation*}
Moreover, for any fixed $x$, the function $\phi(x,y)$ is $L^{\infty}$ in $y$ and can be estimated as
\begin{equation}\label{phi-upper-bound}
\sup_{y \in \R^2} |\phi(x,y)| \leq \frac{C}{1+t}\Big(1 + \log(1+t)+\frac{|x|^2}{1+t}\Big).
\end{equation}
\end{lemma}
\begin{proof}
The two-side cancellation conditions are easily verified by  direct computations, and we shall omit the proof here. Now we prove the upper bound \eqref{phi-upper-bound} by estimating each term in the test function $\phi(x,y)$ individually. We mention that the constant $C$ will change line by line for convenience in the following. For the singular term $w_1(z)$ in the test function \eqref{LDP:vector field}, we have by integration by parts and the long time asymptotic decay Lemma \ref{lemma2} that
\begin{equation*}
\begin{aligned}
|w_1(x)| & = \bigg| \int_{\R^2} \Big(\nabla \log \bar{\rho}(z) + \beta K \ast \bar{\rho}(z)\Big) \cdot K(x-z) \bar\rho(z) \ud z \bigg| \\ & \leq \bigg| \int_{\R^2} \Big(\nabla \log \bar{\rho}(z)\Big) \cdot K(x-z) \bar{\rho}(z) \ud z\bigg| + \bigg| \int_{\R^2} \beta K \ast \bar{\rho}(z) \cdot K(x - z) \bar{\rho}(z) \ud z \bigg| \\ & \leq |\nabla K \ast \bar\rho(x)|+ \frac{C}{\sqrt{1+t}}\int_{\R^2} |K(x - z)| \bar{\rho}(z) \ud z\\ & \leq \frac{C}{1+t}+ \frac{C}{\sqrt{1+t}}\Big(\frac{1}{1+t}\int_{|x-z| \leq \sqrt{1+t}} \frac{1}{|x-z|} \ud z+\int_{|x-z| \geq \sqrt{1+t}} \frac{\bar\rho(z)}{\sqrt{1+t}} \ud z \Big) \\ & \leq \frac{C}{1+t}.
\end{aligned}
\end{equation*}
Since the above bound is independent of $x$, we also have $|w_1(y)| \leq \frac{C}{1+t}$. For the regular term $w_2(z)$ in the test function \eqref{LDP:vector field}, we have by the logarithmic gradient estimate Theorem \ref{Thm:Logarithm gradient estimate} and the long time asymptotic decay Lemma \ref{lemma2} and the Gaussian upper bound Lemma \ref{Pro: Gaussian upper bound} that
\begin{equation}\label{w_2-upper-bound}
\begin{aligned}
|w_2(x)| & = \Big|\Big(\nabla \log \bar{\rho}(x) + \beta K \ast \bar{\rho}(x)\Big) \cdot K \ast \bar{\rho}(x)\Big| \\ & \leq |\nabla \log \bar{\rho}(x) \cdot K \ast \bar{\rho}(x)| + \beta |K \ast \bar{\rho}(x)|^2 \\ & \leq \frac{C}{1+t}\Big(1+\sqrt{\log(1+t)}\Big) +\frac{C}{1+t} \int_{\R^2} \frac{|x|}{|x-z|}\bar\rho(z) \ud z \\
& \leq \frac{C}{1+t}\Big(1+\sqrt{\log(1+t)}\Big) +\frac{C}{1+t} \Big(\sup_{z} |z|\bar\rho(z) \int_{|x-z| \leq \sqrt{1+t}} \frac{1}{|x-z|} \ud z+\int_{|x-z| \geq \sqrt{1+t}} |z|\bar\rho(z) \ud z \Big)\\
& \leq \frac{C}{1+t}\Big(1+\sqrt{\log(1+t)}\Big).
\end{aligned}
\end{equation}
Since the above bound is independent of $x$, we also have $|w_2(y) \leq \frac{C}{1+t}\Big(1+\sqrt{\log(1+t)}\Big)$. For the constant term $c = \int_{\R^2} w(z) \cdot K \ast \bar{\rho}(z) \bar{\rho}(z) \ud z$ in the test function \eqref{LDP:vector field}, we directly use the previous estimate for $w_2(z)$ \eqref{w_2-upper-bound} to obtain that
\begin{equation*}
 c=\int_{\R^2} w(z) \cdot K \ast \bar{\rho}(z) \bar{\rho}(z) \ud z=\int_{\R^2} w_2(z) \bar\rho(z) \ud z \leq \frac{C}{1+t}\Big(1+\sqrt{\log(1+t)}\Big),
\end{equation*}
which also shows that the constant $c$ is finite. For the commutator term $\varphi(x,y)=\big(w(x) - w(y)\big) \cdot K(x - y)$ in the test function \eqref{LDP:vector field}, due to the singularity of $K(x-y)$ on the line $x=y$, we shall deal with the singular case $|x - y| \leq \sqrt{1+t}$ and the regular case $|x - y| \geq \sqrt{1+t}$ respectively. For the singular case $|x - y| \leq \sqrt{1+t}$, we use the mean value inequality to cancel the singularity, together with the long time asymptotic decay Lemma \ref{lemma2} and the logarithmic Hessian estimate Theorem \ref{Thm:Logarithm Hessian estimate}, to obtain that 
\begin{align*}
\Big|\Big(w(x) - w(y)\Big) \cdot K(x-y)\Big|
\leq&\, C\sup_{|z - x| \leq \sqrt{1+t}} \Big(|\nabla^2 \log \bar{\rho}(z)|+\beta |\nabla^2 g \ast\bar\rho(z)|\Big)\\
\leq&\, C\sup_{|z - x| \leq \sqrt{1+t}} \Big(\frac{1+\log(1+t)}{1+t}+\frac{|z|^2}{(1+t)^2}\Big)\\
\leq&\, \frac{C}{1+t}\Big(1 + \log(1+t) + \frac{|x|^2}{1+t}\Big), 
\end{align*}
while for the regular case $|x-y| \geq \sqrt{1+t}$ we use the logarithmic gradient estimate Theorem \ref{Thm:Logarithm gradient estimate} and the long time asymptotic decay Lemma \ref{lemma2} and the Cauchy-Schwartz inequality to obtain that
\begin{align*}
|(w(x) - w(y)) \cdot K(x-y)| &\leq C \frac{|w(x)| + |w(y)|}{|y - x|}\\
&\leq \frac{C}{|y-x|}\Big(\frac{1 + \sqrt{\log(1+t)}}{\sqrt{1+t}} + \frac{|x|}{1+t} + \frac{|y|}{1+t}\Big)\\
&\leq \frac{C}{\sqrt{1+t}}\Big(\frac{1 + \sqrt{\log(1+t)}}{\sqrt{1+t}} + \frac{|x|}{1+t}\Big)+\frac{C}{1+t}\frac{|y-x|+|x|}{|y-x|}\\
&\leq \frac{C}{1+t}\Big(1 + \sqrt{\log(1+t)} +\frac{|x|}{\sqrt{1+t}}\Big)\\
&\leq \frac{C}{1+t}\Big(1 + \log(1+t) + \frac{|x|^2}{1+t}\Big).
\end{align*}
Collecting all estimates above, we finish the proof.
\end{proof}

Now we come back to the time derivative of modulated free energy estimate \eqref{LDP:wait-eta}. We denote that
\begin{equation*}
\begin{aligned}
\psi(t,x) = \frac{C}{1+t}\Big(1 + \log(1+t) + \frac{|x|^2}{1+t}\Big)
\end{aligned}
\end{equation*}
for our upper bound for the test function.
Recall from \cite[Section 1.4]{jabin2016mean} that the condition
\begin{equation*}
\sup_{p \geq 1} \frac{\|\varphi\|_{L^p(\bar{\rho} \ud x)}}{p} < \infty
\end{equation*}
is equivalent to the exponential integrability that there exists some $\lambda > 0$ such that
\begin{equation*}
\int_{\R^2} e^{\lambda \varphi}  \bar{\rho}(x) \ud x< \infty.
\end{equation*}
Moreover, as pointed out in \cite{feng2023quantitative}, we have the quantitative relation
\begin{equation*}
\sup_{p \geq 1} \frac{\|\varphi\|_{L^p(\bar{\rho} \ud x)}}{p} \leq \frac{1}{\lambda} \int_{\R^2} e^{\lambda \varphi} \bar{\rho} \ud x.
\end{equation*}
Recall the Gaussian upper bound Lemma \ref{Pro: Gaussian upper bound}
\begin{equation*}
\bar{\rho}(x,t) \le \frac{C}{1+t}\exp\Big(-\frac{|x|^2}{8t+C}\Big),
\end{equation*}
which inspires us to choose the exponential weight $\lambda(t) = \varepsilon (1+t)$ with $\varepsilon > 0$ small enough, we have
\begin{equation*}
\begin{aligned}
\sup_{p \geq 1} \frac{\|\varphi\|_{L^p(\bar{\rho} \ud x)}}{p} \leq \frac{1}{\lambda} \int_{\R^2} e^{\lambda \varphi} \ud \bar{\rho}(x) & \leq \frac{1}{\varepsilon (1+t)^2} \int_{\R^2} \exp\Big(C\varepsilon\big(1 + \log(1+t) + \frac{|x|^2}{1+t}\big) -\frac{|x|^2}{8t+C}\Big) \ud x \\ & \leq \frac{e^{C\varepsilon}(1+t)^{C\varepsilon}}{\varepsilon (1+t)^2} \int_{\R^2} \exp\Big(\frac{C\varepsilon|x|^2}{1+t}-\frac{|x|^2}{8t+C}\Big)  \ud x \\ & \leq \frac{C}{\varepsilon(1+t)^{1-\varepsilon}}.
\end{aligned}
\end{equation*}
Now we choose suitable $\eta(t)$ to satisfy the condition \eqref{LDP:condition2}. It suffices to take 
\begin{equation*}
\eta(t) = C \varepsilon (1+t)^{1 - \varepsilon}.
\end{equation*}
Plugging it into the time evolution inequality of modulated free energy, we have
\begin{equation}
\begin{split}
E_N(\rho_N|\bar{\rho}_N) \leq &\, E_N(\rho_N|\bar{\rho}_N)(0) + \int_0^t \frac{1}{\eta(s)} \mathcal{H}_N(\rho_N|\bar{\rho}_N)(s) \ud s \\ & + \frac{1}{N} \int_0^t \frac{1}{\eta(s)} \log \int_{\R^{2N}} \exp{(N\eta \Phi_N)}d\bar{\rho}_N \ud s \\ \leq &\, E_N(\rho_N|\bar{\rho}_N)(0) + \int_0^t \frac{C}{\varepsilon (1+s)^{1-\varepsilon}} E_N(\rho_N|\bar{\rho}_N)(s) \ud s + \frac{C (1+t)^{\varepsilon}}{\varepsilon N}.
\end{split}
\end{equation}
for any small $\varepsilon > 0$ and some constant $C = C(\beta, \bar{\rho}_0) > 0$. Applying the generalized Gr\"onwall's inequality in \cite[Theorem A]{YE20071075} with
\begin{equation*}
    x(t)=E_N(\rho_N|\bar\rho_N)(t), \quad h(t)=E_N(\rho_N|\bar\rho_N)(0)+\frac{C(1+t)^\varepsilon}{\varepsilon N}, \quad k(t)=\frac{C}{\varepsilon(1+t)^{1-\varepsilon}},
\end{equation*}
we conclude our desired estimate.

\section{Appendix}\label{Appendix}

\subsection{Proof of Lemma \ref{Lemma:LGE}}
Directly by \cite[Section 6]{feng2023quantitative}, we have
\begin{equation*}
\begin{aligned}
\Delta \Big( \frac{|\nabla \bar{\rho}|^2}{\bar{\rho}}\Big) = &  2\frac{|\nabla^2 \bar{\rho}|^2}{\bar{\rho}} + 2\frac{|\nabla \bar{\rho}|^4}{\bar{\rho}^3} - 4\frac{\nabla \bar{\rho} \cdot \nabla^2 \bar{\rho} \cdot \nabla \bar{\rho}}{\bar{\rho}^2} + \frac{2 \nabla \bar{\rho}} {\bar{\rho}} \cdot \nabla \Delta \bar{\rho} - \frac{|\nabla \bar{\rho}|^2}{\bar{\rho}^2} \Delta \bar{\rho} \\ = & \frac{2}{\bar{\rho}}\bigg|\nabla^2 \bar{\rho} - \frac{\nabla \bar{\rho} \otimes \nabla \bar{\rho}}{\bar{\rho}}\bigg|^2 + \frac{2 \nabla \bar{\rho}} {\bar{\rho}} \cdot \nabla \Delta \bar{\rho} - \frac{|\nabla \bar{\rho}|^2}{\bar{\rho}^2} \Delta \bar{\rho},
\end{aligned}
\end{equation*}

\begin{equation*}
\begin{aligned}
\partial_t \Big( \frac{|\nabla \bar{\rho}|^2}{\bar{\rho}}\Big) = & \frac{2 \nabla \bar{\rho}} {\bar{\rho}} \cdot \nabla (\Delta \bar{\rho} - \nabla g \ast \bar{\rho} \cdot \nabla \bar{\rho} - q \bar{\rho}) - \frac{|\nabla \bar{\rho}|^2}{\bar{\rho}^2} (\Delta \bar{\rho} - \nabla g \ast \bar{\rho} \cdot \nabla \bar{\rho} - q \bar{\rho}) \\ = & \frac{2 \nabla \bar{\rho}} {\bar{\rho}} \cdot \nabla \Delta \bar{\rho} - \frac{|\nabla \bar{\rho}|^2}{\bar{\rho}^2} \Delta \bar{\rho} - \nabla g \ast \bar{\rho} \cdot \nabla \Big( \frac{|\nabla \bar{\rho}|^2}{\bar{\rho}}\Big) - q\Big( \frac{|\nabla \bar{\rho}|^2}{\bar{\rho}}\Big) \\ & - \frac{2 \nabla \bar{\rho}}{\bar{\rho}} \cdot \nabla^2 g \ast \bar{\rho} \cdot \nabla \bar{\rho} - 2 \nabla \bar{\rho} \cdot \nabla q. 
\end{aligned}
\end{equation*}
Using the Cauchy-Schwartz inequality and the boundedness of $\nabla q$, we have 
\[ \nabla \bar{\rho} \cdot \nabla q \leq C'_1 \frac{|\nabla \bar{\rho}|^2}{\bar{\rho}} + C'_2 \bar{\rho}, \]
for some $C'_1, C'_2 > 0$. Then we have
\begin{equation}
\begin{aligned}
(\partial_t + q - \Delta_f) \Big( \frac{|\nabla \bar{\rho}|^2}{\bar{\rho}}\Big) = & -\frac{2}{\bar{\rho}}\bigg|\nabla^2 \bar{\rho} - \frac{\nabla \bar{\rho} \otimes \nabla \bar{\rho}}{\bar{\rho}}\bigg|^2 - \frac{2 \nabla \bar{\rho}}{\bar{\rho}} \cdot \nabla^2 g \ast \bar{\rho} \cdot \nabla \bar{\rho} - 2 \nabla \bar{\rho} \cdot \nabla q \\ \leq &\, C_1 \frac{|\nabla \bar{\rho}|^2}{\bar{\rho}} + C_2 \bar{\rho}.
\end{aligned}
\end{equation}
where $C_1 = C_1(\|\nabla^2 g \ast \bar{\rho}\|_{L^{\infty}}, \|\nabla q\|_{L^{\infty}}), C_2 = C_2(\| \nabla q\|_{L^{\infty}})$. Moreover, if we take $q = \bar{\rho}$, we have
\begin{equation*}
\begin{aligned}
\partial_t \Big( \frac{|\nabla \bar{\rho}|^2}{\bar{\rho}}\Big) = &\, \frac{2 \nabla \bar{\rho}} {\bar{\rho}} \cdot \nabla (\Delta \bar{\rho} - \nabla g \ast \bar{\rho} \cdot \nabla \bar{\rho} - q \bar{\rho}) - \frac{|\nabla \bar{\rho}|^2}{\bar{\rho}^2} (\Delta \bar{\rho} - \nabla g \ast \bar{\rho} \cdot \nabla \bar{\rho} - q \bar{\rho}) \\ = &\, \frac{2 \nabla \bar{\rho}} {\bar{\rho}} \cdot \nabla \Delta \bar{\rho} - \frac{|\nabla \bar{\rho}|^2}{\bar{\rho}^2} \Delta \bar{\rho} - \nabla g \ast \bar{\rho} \cdot \nabla \Big( \frac{|\nabla \bar{\rho}|^2}{\bar{\rho}}\Big) - q\Big( \frac{|\nabla \bar{\rho}|^2}{\bar{\rho}}\Big) \\ & - \frac{2 \nabla \bar{\rho}}{\bar{\rho}} \cdot \nabla^2 g \ast \bar{\rho} \cdot \nabla \bar{\rho} - 2 \nabla \bar{\rho} \cdot \nabla q \\ \leq &\, \frac{2 \nabla \bar{\rho}} {\bar{\rho}} \cdot \nabla \Delta \bar{\rho} - \frac{|\nabla \bar{\rho}|^2}{\bar{\rho}^2} \Delta \bar{\rho} - \nabla g \ast \bar{\rho} \cdot \nabla \Big( \frac{|\nabla \bar{\rho}|^2}{\bar{\rho}}\Big) - q \Big( \frac{|\nabla \bar{\rho}|^2}{\bar{\rho}} \Big) - \frac{2 \nabla \bar{\rho}}{\bar{\rho}} \cdot \nabla^2 g \ast \bar{\rho} \cdot \nabla \bar{\rho}.
\end{aligned}
\end{equation*}
Then we conclude that
\begin{equation}
\begin{aligned}
(\partial_t + q - \Delta_f) \Big( \frac{|\nabla \bar{\rho}|^2}{\bar{\rho}}\Big) \leq & -\frac{2}{\bar{\rho}}\bigg|\nabla^2 \bar{\rho} - \frac{\nabla \bar{\rho} \otimes \nabla \bar{\rho}}{\bar{\rho}}\bigg|^2 - \frac{2 \nabla \bar{\rho}}{\bar{\rho}} \cdot \nabla^2 g \ast \bar{\rho} \cdot \nabla \bar{\rho} \\ \leq &\, C \frac{|\nabla \bar{\rho}|^2}{\bar{\rho}},
\end{aligned}
\end{equation}
where $C = C(\| \nabla^2 g \ast \bar{\rho}\|_{L^{\infty}})$.
Similarly, we have
\begin{equation*}
\Delta (\bar{\rho} \log \bar{\rho}) = (1 + \log \bar{\rho}) \Delta \bar{\rho} + \frac{|\nabla \bar{\rho}|^2}{\bar{\rho}}, 
\end{equation*}
\begin{equation*}
\begin{aligned}
\partial_t (\bar{\rho} \log \bar{\rho}) = & (1 + \log \bar{\rho})(\Delta \bar{\rho} - \nabla g \ast \bar{\rho} \cdot \nabla \bar{\rho} - q \bar{\rho}) \\ = & (1 + \log \bar{\rho})\Delta \bar{\rho} - \nabla g \ast \bar{\rho} \cdot \nabla(\bar{\rho} \log \bar{\rho}) - q \bar{\rho} \log \bar{\rho} - q\bar{\rho}, 
\end{aligned}
\end{equation*}
then we obtain
\begin{equation}
\begin{aligned}
(\partial_t + q - \Delta_f)(\bar{\rho} \log \bar{\rho}) & = -\frac{|\nabla \bar{\rho}|^2}{\bar{\rho}} - q \bar{\rho} \leq -\frac{|\nabla \bar{\rho}|^2}{\bar{\rho}},
\end{aligned}
\end{equation}
by the non-negativity of $q$. Finally, by the Cauchy-Schwartz inequality and the boundedness of $q_t$, we have
\begin{equation}
\begin{aligned}
(\partial_t + q - \Delta_f)\rho_t = - \nabla g \ast \rho_t \cdot \nabla \rho - q_t \bar{\rho} \leq C_3 \frac{|\nabla \bar{\rho}|^2}{\bar{\rho}} + C_3' \bar{\rho}. 
\end{aligned}
\end{equation}

\subsection{Proof of Lemma \ref{Lemma:LHE}}
For the first line of Lemma \ref{Lemma:LHE}, we rewrite \eqref{Lemma:LGE-1} in Lemma \ref{Lemma:LGE} as following,
\begin{equation*}
\begin{aligned}
(\partial_t + q - \Delta_f )(\frac{|\nabla \bar{\rho}|^2}{\bar{\rho}}) & = - \frac{2}{\bar{\rho}} |\nabla_i \nabla_j {\bar{\rho}} - \frac{\nabla_i \rho \nabla_j {\bar{\rho}}}{{\bar{\rho}}}|^2 - \frac{2}{{\bar{\rho}}} \nabla {\bar{\rho}} \cdot \nabla (K \ast {\bar{\rho}}) \cdot \nabla {\bar{\rho}} - 2|\nabla {\bar{\rho}}|^2, \\ & \leq - \frac{2}{{\bar{\rho}}} |\nabla_i \nabla_j {\bar{\rho}} - \frac{\nabla_i {\bar{\rho}} \nabla_j {\bar{\rho}}}{{\bar{\rho}}}|^2 - \frac{2}{{\bar{\rho}}} \nabla {\bar{\rho}} \cdot \nabla (K \ast {\bar{\rho}}) \cdot \nabla {\bar{\rho}} \\ & = - \frac{2}{{\bar{\rho}}}|\nabla^2 {\bar{\rho}}|^2 - \frac{2 |\nabla {\bar{\rho}}|^4}{{\bar{\rho}}^3} + \frac{4}{{\bar{\rho}}^2} \nabla^2 {\bar{\rho}}: \nabla {\bar{\rho}} \otimes \nabla {\bar{\rho}} - \frac{2}{{\bar{\rho}}} \nabla {\bar{\rho}} \cdot \nabla (K \ast {\bar{\rho}}) \cdot \nabla {\bar{\rho}} \\ & \leq - \frac{|\nabla^2 {\bar{\rho}}|^2}{{\bar{\rho}}} + \frac{2 |\nabla {\bar{\rho}}|^4}{{\bar{\rho}}^3} + \frac{C}{t \vee 1}\frac{|\nabla {\bar{\rho}}|^2}{{\bar{\rho}}}.
\end{aligned}
\end{equation*}
For the second line of Lemma \ref{Lemma:LHE}, directly by \cite[Section 7]{feng2023quantitative}, we have
\begin{equation*}
\begin{aligned}
\Delta \Big( \frac{|\nabla^2 \bar{\rho}|^2}{\bar{\rho}}\Big) = &\,  2\frac{|\nabla^3 \bar{\rho}|^2}{\bar{\rho}} + 2\frac{|\nabla^2 \bar{\rho}|^2 |\nabla \bar{\rho}|^2}{\bar{\rho}^3} - 4\frac{\nabla \bar{\rho} \cdot \nabla^2 \bar{\rho} \cdot \nabla \bar{\rho}}{\bar{\rho}^2} + \frac{2 \nabla^2 \bar{\rho}} {\bar{\rho}} \cdot \nabla^2 \Delta \bar{\rho} - \frac{|\nabla^2 \bar{\rho}|^2}{\bar{\rho}^2} \Delta \bar{\rho} \\ = &\, \frac{2}{\bar{\rho}}\bigg|\nabla^3 \bar{\rho} - \frac{\nabla \bar{\rho} \otimes \nabla^2 \bar{\rho}}{\bar{\rho}}\bigg|^2 + \frac{2 \nabla^2 \bar{\rho}} {\bar{\rho}} \cdot \nabla^2 \Delta \bar{\rho} - \frac{|\nabla^2 \bar{\rho}|^2}{\bar{\rho}^2} \Delta \bar{\rho},
\end{aligned}
\end{equation*}
and 
\begin{equation*}
\begin{aligned}
\partial_t \Big( \frac{|\nabla^2 \bar{\rho}|^2}{\bar{\rho}}\Big) = & \,\frac{2 \nabla^2 \bar{\rho}} {\bar{\rho}} \cdot \nabla^2 (\Delta \bar{\rho} - \nabla g \ast \bar{\rho} \cdot \nabla \bar{\rho} - q \bar{\rho}) - \frac{|\nabla^2 \bar{\rho}|^2}{\bar{\rho}^2} (\Delta \bar{\rho} - \nabla g \ast \bar{\rho} \cdot \nabla \bar{\rho} - q \bar{\rho}) \\ = &\, \frac{2 \nabla^2 \bar{\rho}} {\bar{\rho}} \cdot \nabla^2 \Delta \bar{\rho} - \frac{|\nabla^2 \bar{\rho}|^2}{\bar{\rho}^2} \Delta \bar{\rho} - \nabla g \ast \bar{\rho} \cdot \nabla \Big( \frac{|\nabla^2 \bar{\rho}|^2}{\bar{\rho}}\Big) - q\Big( \frac{|\nabla^2 \bar{\rho}|^2}{\bar{\rho}}\Big) \\ & - \frac{2 \nabla^2 \bar{\rho}}{\bar{\rho}} \cdot \nabla^3 g \ast \bar{\rho} \cdot \nabla \bar{\rho} - \frac{4 \nabla^2 \bar{\rho}}{\bar{\rho}} \cdot \nabla^2 g \ast \bar{\rho} \cdot \nabla^2 \bar{\rho} - 2 \nabla^2 \bar{\rho} \cdot \nabla^2 q - \frac{4 \nabla^2 \bar{\rho}}{\bar{\rho}}:\nabla q \otimes \nabla \bar{\rho}.
\end{aligned}
\end{equation*}
Using the Cauchy-Schwartz inequality and the
boundedness of $\nabla^2 q$ and $\nabla q$, we have
\begin{equation*}
\begin{aligned}
- \frac{2 \nabla^2 \bar{\rho}}{\bar{\rho}} \cdot \nabla^3 g \ast \bar{\rho} \cdot \nabla \bar{\rho} & \leq \frac{C}{(t \vee 1)^2} \frac{|\nabla^2 \bar{\rho}|^2}{\bar{\rho}} + \frac{C'}{(t \vee 1)^2} \frac{|\nabla \bar{\rho}|^2}{\bar{\rho}}, \\ 
- 2 \nabla^2 \bar{\rho} \cdot \nabla^2 q & \leq \frac{C}{(t \vee 1)^2} \frac{|\nabla^2 \bar{\rho}|^2}{\bar{\rho}} + \frac{C'}{(t \vee 1)^2} \bar{\rho}, \\
- \frac{4 \nabla^2 \bar{\rho}}{\bar{\rho}}:\nabla q \otimes \nabla \bar{\rho} & \leq \frac{C}{t \vee 1} \frac{|\nabla^2 \bar{\rho}|^2}{\bar{\rho}} + \frac{C'}{t \vee 1} \frac{|\nabla \bar{\rho}|^2}{\bar{\rho}},
\end{aligned}
\end{equation*}
for some $C, C' > 0$, then 
\begin{equation*}
\begin{aligned}
(\partial_t + q - \Delta_f) \Big( \frac{|\nabla^2 \bar{\rho}|^2}{\bar{\rho}}\Big) = & -\frac{2}{\bar{\rho}}\bigg|\nabla^3 \bar{\rho} - \frac{\nabla \bar{\rho} \otimes \nabla^2 \bar{\rho}}{\bar{\rho}}\bigg|^2 - 2 \nabla^2 \bar{\rho} \cdot \nabla^2 q - \frac{4 \nabla^2 \bar{\rho}}{\bar{\rho}}:\nabla q \otimes \nabla \bar{\rho}\\ & - \frac{2 \nabla^2 \bar{\rho}}{\bar{\rho}} \cdot \nabla^3 g \ast \bar{\rho} \cdot \nabla \bar{\rho} - \frac{4 \nabla^2 \bar{\rho}}{\bar{\rho}} \cdot \nabla^2 g \ast \bar{\rho} \cdot \nabla^2 \bar{\rho} \\ \leq & \, \frac{C_1}{t \vee 1} \frac{|\nabla^2 \bar{\rho}|^2}{\bar{\rho}} + \frac{C_1'}{t \vee 1} \frac{|\nabla \bar{\rho}|^2}{\bar{\rho}} + \frac{C_1''}{(t \vee 1)^2} \bar{\rho}.
\end{aligned}
\end{equation*}
Moreover, if we take $q = \bar{\rho}$, we have
\begin{equation*}
- 2 \nabla^2 \bar{\rho} \cdot \nabla^2 q = - 2 \nabla^2 \bar{\rho} \cdot \nabla^2 \bar{\rho} = - 2 |\nabla^2 \bar{\rho}|^2 \leq 0,
\end{equation*}
and 
\begin{equation*}
- \frac{4 \nabla^2 \bar{\rho}}{\bar{\rho}}:\nabla q \otimes \nabla \bar{\rho} = - \frac{4 \nabla^2 \bar{\rho}}{\bar{\rho}}:\nabla \bar{\rho} \otimes \nabla \bar{\rho} \leq \frac{C'}{(t \vee 1)^2} \frac{|\nabla \bar{\rho}|^2}{\bar{\rho}},
\end{equation*}
hence we finally have
\begin{equation}
\begin{aligned}
(\partial_t + q - \Delta_f) \Big( \frac{|\nabla^2 \bar{\rho}|^2}{\bar{\rho}}\Big) = & -\frac{2}{\bar{\rho}}\bigg|\nabla^3 \bar{\rho} - \frac{\nabla \bar{\rho} \otimes \nabla^2 \bar{\rho}}{\bar{\rho}}\bigg|^2 - 2 \nabla^2 \bar{\rho} \cdot \nabla^2 q - \frac{4 \nabla^2 \bar{\rho}}{\bar{\rho}}:\nabla q \otimes \nabla \bar{\rho}\\ & - \frac{2 \nabla^2 \bar{\rho}}{\bar{\rho}} \cdot \nabla^3 g \ast \bar{\rho} \cdot \nabla \bar{\rho} - \frac{4 \nabla^2 \bar{\rho}}{\bar{\rho}} \cdot \nabla^2 g \ast \bar{\rho} \cdot \nabla^2 \bar{\rho} \\ \leq & \frac{C_1}{t \vee 1
} \frac{|\nabla^2 \bar{\rho}|^2}{\bar{\rho}} + \frac{C_1'}{(t \vee 1)^2} \frac{|\nabla \bar{\rho}|^2}{\bar{\rho}}.
\end{aligned}
\end{equation}

Next we focus on the propagation of $\bar{\rho}(\log \bar{\rho})^2$, again by \cite[Section 7]{feng2023quantitative}, 
\begin{equation*}
\begin{aligned}
\Delta(\bar{\rho}(\log \bar{\rho})^2) = \Delta \rho(\log \bar{\rho})^2 + 2\Delta \rho(\log \bar{\rho}) + 2 \frac{|\nabla \bar{\rho}|^2}{\bar{\rho}} + 2 \frac{|\nabla \bar{\rho}|^2}{\bar{\rho}} \log \bar{\rho}.
\end{aligned}
\end{equation*}
Hence we finally have
\begin{equation*}
\begin{aligned}
\partial_t (\bar{\rho}(\log \bar{\rho})^2) = \Delta \rho(\log \bar{\rho})^2 + 2\Delta \rho(\log \bar{\rho}) - \nabla g \ast \bar{\rho} \cdot \nabla (\bar{\rho}(\log \bar{\rho})^2) - q \bar{\rho}(\log \bar{\rho})^2 - 2q \bar{\rho} \log \bar{\rho}, 
\end{aligned}
\end{equation*}
\begin{equation*}
(\partial_t + q - \Delta_f)(\bar{\rho} (\log \bar{\rho})^2) = - \frac{2|\nabla \bar{\rho}|^2}{\bar{\rho}}(1 + \log \bar{\rho}) - 2q \bar{\rho} \log \bar{\rho}.
\end{equation*}

\subsection*{Acknowledgements}
This work was partially supported by  the National Key R\&D Program of China, Project Number 2021YFA1002800, NSFC grant No.12171009 and Young
Elite Scientist Sponsorship Program by China Association for Science and Technology (CAST) No. YESS20200028. Z.W. would like to thank Prof. Lei Li from SJTU for discussing about the small interaction force case in Remark \ref{rekConvex}, that leads to our final proof of Theorem \ref{CRCon}.

\bibliographystyle{alpha}
\bibliography{sample}

\end{document}